\documentclass[12pt]{amsart}
 
\usepackage{mathtools}
\mathtoolsset{showonlyrefs,showmanualtags}
 
\usepackage{amssymb, amsmath, amsthm, esint}
\usepackage{mathrsfs}
\usepackage{bbm, dsfont}
\usepackage{color}

\usepackage[top=1.5in, bottom=1.5in, left=1.25in, right=1.25in]{geometry}
\usepackage[all]{xy}
\usepackage{amscd}
 
\usepackage[bottom]{footmisc}
 
\usepackage{fancyhdr}
 
\AtBeginDocument{\setlength{\footskip}{30pt}}
\fancypagestyle{plain}{%
  \fancyhf{}%
  \fancyfoot[C]{\thepage}%
}
 
\fancypagestyle{myfancy}{%
  \fancyhf{}%
  \fancyhead[CE]{\scshape A Variation Norm Carleson Theorem Along the Primes}%
  \fancyhead[CO]{A. Fragkos, B. Krause, N. Miheisi, and Y.-C. Sun}%
  \fancyhead[LE,RO]{\thepage}%
}
 
 \def\supp {{\mathrm{supp}\,}}
\numberwithin{equation}{section}

\newtheorem{theorem}{Theorem}[section]
\newtheorem*{theorem*}{Theorem}

\newtheorem{proposition}[theorem]{Proposition}
\newtheorem*{proposition*}{Proposition}
\newtheorem{cor}[theorem]{Corollary}
\newtheorem{lemma}[theorem]{Lemma}

\theoremstyle{definition}
\newtheorem{problem}{Problem}
 
\usepackage[
  pagebackref,
  colorlinks=true,
  linkcolor=blue,
  citecolor=magenta,
  filecolor=magenta,
  urlcolor=cyan
]{hyperref}

\makeatletter
\renewcommand*{\backref}[1]{}
\renewcommand*{\backrefalt}[4]{%
  \ifcase #1\relax
  \or
    \space (Page~#2)%
  \else
    \space (Pages~#2)%
  \fi
}
\makeatother

\begin{document}
 
\title{A Variation Norm Carleson Theorem Along the Primes}
 

\author{Anastasios Fragkos}
\address{
Department of Mathematics,
Georgia Tech \\
686 Cherry Street, Atlanta, GA 30332-0160 USA}
\email{anastasiosfragkos@gatech.edu}

\author{Ben Krause}
\address{
Department of Mathematics,
University of Bristol \\
Beacon House, Queens Rd, Bristol BS8 1QU}
\email{ben.krause@bristol.ac.uk}

\author{Nazar Miheisi}
\address{
Department of Mathematics,
King's College London \\
Strand Building, Strand Campus, Strand, London, WC2R 2LS}
\email{naz.miheisi@gmail.com}

\author{Yu-Chen Sun}
\address{
Department of Mathematics,
University of Bristol \\
Beacon House, Queens Rd, Bristol BS8 1QU}
\email{yuchensun93@163.com}
 
\date{\today}
 
\begin{abstract}
Let $\Lambda$ denote the von Mangoldt function; we prove that for each $r > 2$, there exist constants
\[ r' < \mathbf{c}(r) < 2 < \mathbf{C}(r), \qquad \lim_{r \to \infty} \mathbf{c}(r) = 1, \ \lim_{r \to \infty} \mathbf{C}(r) = \infty \]
so that the discrete variational Carleson operator along the primes
\begin{align}
\mathcal{V}^r \Big( \sum_{n \neq 0} f(x-n) \Lambda(|n|)  \frac{e^{2\pi i \lambda n}}{n} : \lambda \in \mathbb{T} \Big)
\end{align}
is bounded on $\ell^p$ for all $\mathbf{c}(r)  < p < \mathbf{C}(r)$, while the variation is unbounded when $p \leq r'$.
At the non-variational endpoint, the same argument gives
the sharp maximal result: the prime Carleson operator
\[
\sup_{\lambda\in\mathbb T}
\Big|\sum_{n\neq0} f(x-n)\Lambda(|n|)\frac{e^{2\pi i\lambda n}}{n}\Big|
\]
is bounded on $\ell^p(\mathbb Z)$ for the optimal range $1<p<\infty$.

The proof gives a new mechanism for treating modulation-invariant singular
integrals after arithmetic sparsification.  It combines higher-order Fourier
uniformity, a variable-coefficient multi-frequency principle in the spirit of
Bourgain, and an additive-combinatorial inverse argument.  A key step is a
reduction to finite periodic models, where the Ramanujan structure of the major
arcs is converted into a sharp estimate for structured atoms by elementary
number theory.
\end{abstract}
 
\maketitle
 
\setcounter{footnote}{0} 
 
\pagestyle{myfancy}
 
\setcounter{tocdepth}{1}
\tableofcontents

\section{Introduction}

The study of maximally modulated singular integrals has a rich history, encompassing arguably the most celebrated theorem of twentieth century Fourier analysis: Carleson's Theorem \cite{C}. We recall the principal estimate in the real-variable setting.
\begin{theorem}[Carleson's Theorem, Real-Variable Formulation]\label{t:CL2}
    The following operator is weakly bounded on $L^2(\mathbb{R})$:
    \begin{align}\label{e:CL2}
        Cf(x):= \sup_{\lambda \in \mathbb{R}} \left| \int f(x-t) \cdot  e^{2 \pi i \lambda t} \ \frac{dt}{t}\right|.
    \end{align}
\end{theorem}
This theorem was later extended to $L^p(\mathbb{R})$ functions by Hunt in \cite{Hunt}, see also \cite{Fef} and \cite{LT} for more modern proofs of Theorem \ref{t:CL2}. And, Carleson's Theorem has admitted profound extensions to the case of maximally polynomially modulated singular integrals, see \cite{Lie} and \cite{ZK}. The major difficulty in addressing $Cf$ is that it is invariant under three symmetries: translation and dilation, like the standard Hilbert transform, but also modulation, in that for any $\theta \in \mathbb{R}$
\begin{align}
    Cf(x) \equiv C(\text{Mod}_\theta f)(x), \; \; \; \text{Mod}_\theta g(x) := e^{2 \pi i \theta x} g(x).
\end{align}
Indeed, the presence of this additional symmetry rules out any approach which treats the e.g.\ zero-frequency as distinguished.

Before the issues of polynomial modulations were addressed, in \cite{O+}, the authors established a profound quantitative strengthening of Theorem \ref{t:CL2}. To state their result, we recall the definition of the $r$ variation of a sequence of scalars,
\begin{align}
    \mathcal{V}^r(a_\lambda : \lambda) := \sup \left( \sum_{i} |a_{\lambda_i} - a_{\lambda_{i+1}}|^r \right)^{1/r}, \qquad 0 < r < \infty
\end{align}
where the supremum runs over all finite increasing subsequences; these operators are pointwise \emph{decreasing} in $r$, with the endpoint given by
\begin{align}
    \mathcal{V}^\infty(a_\lambda) := \sup_{\lambda \neq \lambda'} |a_{\lambda} - a_{\lambda'}|,
\end{align}
namely the diameter of the sequence, typically comparable to $\| a_\lambda \|_{L^\infty(\lambda)}$. For sequences of functions, $\{ f_\lambda \}$, one defines the variation pointwise,
\begin{align}\label{e:fxnvar}
    \mathcal{V}^r(f_\lambda:\lambda)(x) :=     \mathcal{V}^r(f_\lambda(x):\lambda); 
\end{align}
note that whenever $ \mathcal{V}^r(a_\lambda) < \infty$, the sequence $\{a_\lambda\}$ necessarily converges (rapidly), so norm estimates on $\mathcal{V}^r(f_\lambda)$ are a strong way to show almost everywhere convergence of $\{ f_\lambda\}$.

Variational methods entered harmonic analysis through martingale convergence
questions \cite{LE}, and have since become a central tool for quantifying
pointwise convergence of operator families \(\{ f_\lambda=T_\lambda f \}\); see \cite{BK} for a brief survey.  A landmark result in this direction is the variation-norm Carleson theorem of Oberlin, Seeger, Tao, Thiele, and Wright, which has led to substantial subsequent work \cite{AU,BM,U}.

\begin{theorem}[Variational Carleson, Real-Variable Formulation]\label{t:Var}
    Suppose that $r > 2$ and $p > r'$. Then the following operator
    \begin{align}
        V^r f(x) := \mathcal{V}^r \left( \int f(x-t) \cdot  e^{2 \pi i \lambda t} \ \frac{dt}{t} : \lambda \in \mathbb{R} \right)
    \end{align}
    is bounded on $L^p(\mathbb{R})$.
\end{theorem}

This paper studies arithmetic analogues of modulation-invariant singular integrals
on sequence spaces.  The ordinary discrete analogue of Theorem \ref{t:Var} follows
from standard transference, reviewed in Appendix \ref{a:disc}; the new difficulty is
to impose an arithmetic sparsification while retaining modulation invariance.  Our
main example is the variation-norm prime Carleson operator: 
\begin{align}\label{e:Carweight}
    \mathcal{V}^r_{\mathbb{P}} f(x) := \mathcal{V}^r \left(\sum_{n \neq 0} f(x-n) w(n) \frac{e^{2 \pi i \lambda n}}{n} : \lambda \in \mathbb{T} \right),
\end{align}
where here and throughout, we let
\begin{align}\label{e:wweight}
    w(n) := \Lambda(|n|) := \begin{cases} \log p & \text{ if } |n| = p^k \text{ is a power of a prime } \\
    0 & \text{ otherwise} \end{cases}
\end{align}
denote the even extension of the von Mangoldt function.\footnote{For the purposes of $\ell^p$-norm estimates for $1 < p < \infty$, it suffices by density to work with finitely-supported functions; we will implicitly assume that every sequence-space function introduced below is finitely supported.}

An interesting feature of our work is that no Euclidean analogue currently exists: $\mathcal{V}^r_{\mathbb{P}}$ can be viewed as a discrete analogue of Euclidean operators of the form
\begin{align}\label{e:Cmu}
   V^r_\mu f(x) := \mathcal{V}^r \left( \int f(x-t) \cdot e^{2 \pi i \lambda t} \ \frac{d\mu(t)}{t} : \lambda \in \mathbb{R} \right),
\end{align}
where $\mu$ is a singular, even, Ahlfors-David regular measure of full Hausdorff dimension. 

Our choice to begin our study of arithmetic Carleson operators with $\mathcal{V}^r_{\mathbb{P}}$ is motivated by the long history of discrete analogues in harmonic analysis along the primes, dating back to the work of Bourgain-Wierdl \cite{Bp, Wierdl}, who studied the maximal function,
\begin{align}\label{e:maxprime}
    \mathcal{M}_{\mathbb{P}} f(x) := \sup_N \frac{1}{2N+1} \sum_{|n| \leq N} |f(x-n) w(n)|,
\end{align}
motivated by questions of pointwise convergence of ergodic averages along prime times in the dynamical systems context; the singular integral, and variation-in-scale, formulations have been similarly addressed in \cite{MT} and in \cite{MTZK}, and in \cite{CHKLP}, a highly restricted variant of the smaller operator \eqref{e:CP} was studied using different methods to those introduced below. More recently, arithmetic extensions of classical theorems in pointwise ergodic theory related to prime number theory have been developed \cite{F+,FKT1,FKT2}, and the main result of this paper complements this recent work, as the singular integral analogue of \cite{F+}, and as a third example of a discrete, modulation invariant operator, at the interface of prime number theory \cite{FKT1,FKT2}; this modulation invariance distinguishes this work from previous work on oscillatory singular integrals \cite{BKSW,KR1,KR2}. Specifically, we establish the following result, which is asymptotically sharp as $r \to \infty$:

\begin{theorem}\label{t:main}
For each $r > 2$ there exist constants $\mathbf{c}(r), \ \mathbf{C}(r)$ so that
\[ r' < \mathbf{c}(r) < 2 < \mathbf{C}(r), \qquad \lim_{r \to \infty} \mathbf{c}(r) = 1, \ \lim_{r \to \infty} \mathbf{C}(r) = \infty \]
and for each $\mathbf{c}(r) < p < \mathbf{C}(r)$, there exists an absolute constant, $0 < C_{p,r} < \infty$, so that the following estimate holds:
\begin{align}
    \| \mathcal{V}^r_{\mathbb{P}} f \|_{\ell^p} \leq C_{p,r} \| f\|_{\ell^p}.
\end{align}
\end{theorem}

The condition \(p>r'\) is unavoidable: testing on \(f=\delta_0\), the value
at \(x=n\) contains the variation of a single exponential:
\[
    \mathcal{V}^r(e(\lambda n):\lambda\in\mathbb T)\gtrsim_r |n|^{1/r}.
\]
Thus the output has size comparable to
\[
    \Lambda(|n|)|n|^{-1+1/r},
\]
which is in \(\ell^p(\mathbb Z)\) only when \(p>r'\).

The same method also gives the full expected range for the corresponding maximal operator,
\begin{align}\label{e:CP}
\mathcal{C}_{\mathbb{P}} f(x):=\sup_{\lambda\in\mathbb T}
\left|\sum_{n\neq0} f(x-n)\Lambda(|n|)
\frac{e^{2 \pi i \lambda n}}{n}\right|.
\end{align}

\begin{theorem}\label{t:carmax}
For every \(1<p<\infty\) there exists an absolute constant $0 < C_p < \infty$ so that 
\[
\|\mathcal{C}_{\mathbb{P}} f\|_{\ell^p}\leq C_p \|f\|_{\ell^p}.
\]
\end{theorem}
The proof of this theorem is obtained along the way: the only place where the variational argument loses range is the change of variables in Lemma \ref{l:red-var-l2}, and this loss disappears in the non-variational supremum formulation.

\bigskip

We now describe our approach.

\subsection{Proof Overview}\label{ss:sketch}

The proof combines three main ingredients.  The first is a higher-order Fourier
decomposition of the von Mangoldt function, following the circle-method framework
used in recent work on Wiener-Wintner along the primes
\cite{F+}. In that setting, the characteristic nature of the Kronecker factor dictated that the argument should be organized according to spectral statistics of the input function. While our arguments begin similarly to those of \cite{F+}, with higher-order Fourier-analytic reductions, the Kronecker factor is much less relevant in the singular-integral/variational setting. The replacement for spectral organization is a variable-coefficient multi-frequency estimate, in the
spirit of Bourgain's maximal inequality \cite{B2} and its entropy refinements, related to those previously developed and refined in \cite{GZK,KWW}; this
is where the variation-norm Carleson theorem enters.  The third is a finite
periodic model in which the arithmetic structure of Ramanujan sums becomes
explicit.

In particular, after these reductions, the main object is a family of single-scale, maximally
modulated operators whose arithmetic weights have the form
\[
    w_{Q,B}:=\sum_{q\in B}\frac{\mu(q)}{\phi(q)}c_q, \qquad c_q(n) := \sum_{(a,q) =1} e^{2 \pi i an/q}
\]
where \(B\) ranges over blocks of square-free integers in \((Q/2,Q]\), with
\(|B|\approx Q^\kappa\) with $\kappa > 0$ sufficiently small; each \(w_{Q,B}\) is periodic with period
\[
    \mathcal Q_B:=\operatorname{lcm}(q:q\in B).
\]
The analytic lifting step reduces the physical problem to bounding the periodic operator
\begin{align}
    M_Bf(x)
    :=
    \sup_{\lambda\in\mathbb Z/\mathcal Q_B}
    \bigg|
    \frac{1}{\mathcal Q_B}
    \sum_{r\leq\mathcal Q_B}
    f(x-r)w_{Q,B}(r)e^{2\pi i\lambda r/\mathcal{Q}_B}
    \bigg|
\end{align}
on $\ell^2$ with a power-savings in $Q$; this is proved by an inverse theorem and an energy decrement argument.  Roughly speaking, if \(M_Bf\) is large, then \(f\) correlates with an atom
which is constant-frequency along arithmetic progressions:
\begin{align}
    [\mathcal Q_B]\ni a+sd
    \longmapsto
    g(a)e^{2\pi i d\vartheta(a)s/\mathcal Q_B},
    \qquad
    \frac1d\sum_{a\leq d}|g(a)|^2=1,
\end{align}
where \(d\mid q\) for some \(q\in B\).  Iterating this inverse statement reduces the
periodic problem to testing \(M_B\) on such structured atoms.

For structured atoms, the remaining estimate is arithmetic, and one is led to study the kernels
\begin{align}
    K_\theta(u)
    :=
    \frac1D
    \sum_{\ell\leq D}
    w_{Q,B}(d\ell+u)
    e^{2\pi i(d\ell+u)\theta/\mathcal Q_B},
    \qquad
    D:=\mathcal Q_B/d;
\end{align}
a Chinese-remainder computation shows that these kernels vanish except when
\[
    t_\theta:=\frac{D}{(\theta,D)}
    \in
    \bigg\{\frac{q}{(q,d)}:q\in B\bigg\}.
\]
Using this sparsity constraint, we prove that 
\[
    \frac{1}{d} \sum_{u \leq d} \sup_\theta |K_\theta(u)| \lesssim Q^{-1+o(1)},
\]
which allows us to conclude that every structured atom \(\Phi\) satisfies
\[
    \|M_B\Phi\|_{L^2(\mathbb Z/\mathcal Q_B)}
    \lesssim Q^{-1+o(1)},
\]
yielding the power saving needed to sum over denominator blocks.

\medskip

We now describe the organization of the paper.
\subsection{Structure}
\medskip

The following diagram records the main dependencies in the proof.

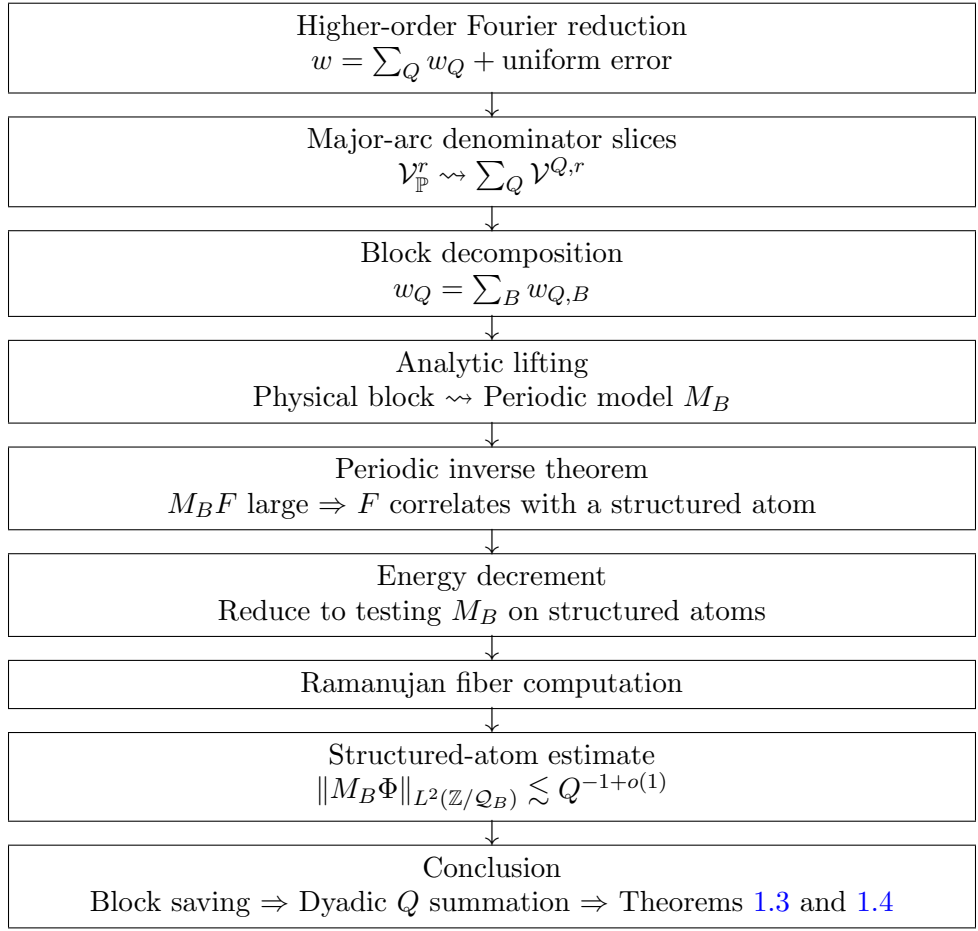
\begin{center}
\small
\setlength{\fboxsep}{4pt}
\begin{tabular}{c}
\fbox{\begin{minipage}{0.82\textwidth}\centering
Higher-order Fourier reduction\\
$w=\sum_Q w_Q+\text{uniform error}$
\end{minipage}}\\[-0.2em]
$\downarrow$\\[-0.2em]
\fbox{\begin{minipage}{0.82\textwidth}\centering
Major-arc denominator slices\\
$\mathcal V_{\mathbb P}^r \leadsto \sum_Q\mathcal V^{Q,r}$
\end{minipage}}\\[-0.2em]
$\downarrow$\\[-0.2em]
\fbox{\begin{minipage}{0.82\textwidth}\centering
Block decomposition\\
$w_Q=\sum_B w_{Q,B}$
\end{minipage}}\\[-0.2em]
$\downarrow$\\[-0.2em]
\fbox{\begin{minipage}{0.82\textwidth}\centering
Analytic lifting\\
Physical block $\leadsto$ Periodic model $M_B$
\end{minipage}}\\[-0.2em]
$\downarrow$\\[-0.2em]
\fbox{\begin{minipage}{0.82\textwidth}\centering
Periodic inverse theorem\\
$M_BF$ large $\Rightarrow$ $F$ correlates with a structured atom
\end{minipage}}\\[-0.2em]
$\downarrow$\\[-0.2em]
\fbox{\begin{minipage}{0.82\textwidth}\centering
Energy decrement\\
Reduce to testing $M_B$ on structured atoms
\end{minipage}}\\[-0.2em]
$\downarrow$\\[-0.2em]
\fbox{\begin{minipage}{0.82\textwidth}\centering
Ramanujan fiber computation
\end{minipage}}\\[-0.2em]
$\downarrow$\\[-0.2em]
\fbox{\begin{minipage}{0.82\textwidth}\centering
Structured-atom estimate\\
$\|M_B\Phi\|_{L^2(\mathbb Z/\mathcal Q_B)}\lesssim Q^{-1+o(1)}$
\end{minipage}}\\[-0.2em]
$\downarrow$\\[-0.2em]
\fbox{\begin{minipage}{0.82\textwidth}\centering
Conclusion\\
Block saving $\Rightarrow$ Dyadic $Q$ summation $\Rightarrow$
Theorems~\ref{t:main} and~\ref{t:carmax}
\end{minipage}}
\end{tabular}
\end{center}

\medskip

With this scheme in mind, we describe the organization of the paper:

\medskip

Section \ref{s:Prelim} collects notation, Gowers-norm estimates, and variational
preliminaries;

Section \ref{s:red} carries out the reduction from the prime operator
to major-arc denominator slices; 

Section \ref{s:MFanalysis} develops the
variable-coefficient multi-frequency estimate used in the physical-to-periodic
reduction;

Section \ref{s:per} introduces the periodic model and proves the inverse
theorem leading to structured atoms; and

Section \ref{s:number} proves the arithmetic
estimate for those atoms;

The appendix contains complementary material: Appendix \ref{a:comp} discusses
related arithmetic weights and the Piatetski-Shapiro extension, while Appendix
\ref{a:disc} recalls the discrete Carleson input used in the paper.

\bigskip

We conclude our introduction by presenting some open questions that follow naturally from our work.

\subsection{Open Problems}
\begin{problem}[Discrete Harmonic Analysis: Arithmetic Testing Conditions]
The proof of Theorem \ref{t:main} uses special structural features of the
von Mangoldt major arcs: square-free denominators, Ramanujan-sum identities, and
near-optimal decay of rational Fourier coefficients.  It is natural to ask for an
intrinsic testing condition on arithmetic weights \(w\) guaranteeing boundedness of
the associated Carleson operators, in analogy with the admissibility criteria in
\cite{FKT1,FKT2}.  Appendix \ref{a:comp} explains why admissibility is insufficient: this problem is quite subtle!
\end{problem}

\begin{problem}[Euclidean Models]
Find natural singular measures \(\mu\) for which the modulation-invariant operators
\begin{align}
   \mathcal{V}^r\bigg(
   \int K(x-t)f(t)e^{2\pi i\lambda t}\,d\mu(t)
   :\lambda\in\mathbb R
   \bigg)
\end{align}
are bounded on \(L^p(\mathbb R)\), even in the case \(p=2\).  The arithmetic weights
appearing in this paper suggest Euclidean analogues in which \(\mu\) is singular,
even, and full-dimensional, while \(K\) is a dimensionally consistent
Calder\'on-Zygmund kernel.
\end{problem}

\begin{problem}[Ergodic Theory: Prime Return Times for Singular Series]
The Wiener-Wintner theorem along the primes was proved in \cite{F+}, while
\cite{FKT1,FKT2} developed prime return-times and double-recurrence theorems.  A
natural singular-integral analogue would be a prime-weighted version of the
return-times theorem for Hilbert-series type averages, in the spirit of
\cite[Theorem 3.4]{D}, \cite[Theorem 3.10]{D}, and \cite[Corollary 3.11]{D}.  One
possible target is almost-sure convergence of
\[
    \sum_{1\leq |n|\leq N}
    \frac{\Lambda(|n|)}{n}
    f(T^n x)g(S^n\omega)
\]
for a full-measure set of return-times parameters \(\omega\), uniformly over
secondary measure-preserving systems.
\end{problem}

\section{Preliminaries}\label{s:Prelim}

\subsection{General Notation}
Throughout, we let
\begin{align}
    e(t) := e^{2 \pi i t}, \; \; \; e_N(t) := e(t/N)
\end{align}
denote the complex exponential,
\begin{align}
    \| x \| := \| x \|_{\mathbb{T}} := \min_{k \in \mathbb{Z}} |x-k|,
\end{align}
and for $\theta \in \mathbb{R}$, define
\begin{align}
    \text{Mod}_\theta g(x) := e(\theta x) g(x);
\end{align}
we identify 
\[ \mathbb{Z}/N := \{ 0,1,\dots, N-1\} \mod N.\]

We use the standard notation for $L^1$-normalized dilations:
\begin{align}
    \varphi_N(n) := \frac{1}{N} \varphi(\frac{n}{N}),
\end{align}
and let
\begin{align}
    M_{\text{HL}} f(x):=  \sup_{I \ni x} \, \frac{1}{|I|} \sum_{n \in I} |f(n)|  
\end{align}
denote the discrete Hardy-Littlewood maximal function; here and throughout, we use $I$ to denote discrete intervals.

We collect
\begin{align}\label{e:GammaQ}
    \Gamma_Q := \{ a/q \text{ reduced} : Q/2 < q \leq Q \},
\end{align}
and define
\begin{align}\label{e:WQ}
    w_Q(n) := \sum_{a/q \in \Gamma_Q} \frac{\mu(q)}{\phi(q)} e(na/q ) = \sum_{Q/2 < q \leq Q} \frac{\mu(q)}{\phi(q)} c_q(n)
\end{align}
where
\begin{align}\label{e:ram}
c_q(n) = \sum_{a/q : \ (a,q) =1} e(a/q n)
\end{align}
denotes a Ramanujan sum, and $\mu,\phi$ denote the M\"{o}bius and totient functions, respectively. We will often make use of the lower bound
\begin{align}\label{e:totientbound}
    \phi(q) \gtrsim \frac{q}{\log \log (10 + q)}.
\end{align}

For subsets $B \subset (Q/2,Q]$
we define
\begin{align}\label{e:lcmQ}
\mathcal{Q}_B := \text{lcm}( q \in B )
\end{align}
and
\begin{align}
w_{Q,B}(n) := \sum_{q \in B} \frac{\mu(q)}{\phi(q)} c_q(n).    
\end{align}

We let $\psi(t) \in \mathcal{C}_c^{\infty}(\{ t : 1/4 < |t| < 1\})$ be a smooth, odd approximation to $\frac{1}{t} \cdot  \mathbf{1}_{|t| \approx 1} $ so that
\begin{align}
\sum_{k \geq 1} \psi_k(t) = \frac{1}{t}, \; \; \; |t| \geq 1, \qquad \psi_k(t) := 2^{-k} \psi(t 2^{-k}).
\end{align}

\subsection{Asymptotic Notation}\label{sss:O}
We will make use of the modified Vinogradov notation. We use $X \lesssim Y$ or $Y \gtrsim X$ to denote
the estimate $X \leq CY$ for an absolute constant $C$ and $X, Y \geq 0.$  If we need $C$ to depend on a
parameter, we shall indicate this by subscripts, thus for instance $X \lesssim_p Y$ denotes the estimate $X \leq C_p Y$ for some $C_p$ depending on $p$. We use $X \approx Y$ as shorthand for $Y \lesssim X \lesssim Y$. We use the notation $X \ll Y$ or $Y \gg X$ to denote that the implicit constant in the $\lesssim$ notation is extremely large, and analogously $X \ll_p Y$ and $Y \gg_p X$. In addition to that, when we write $N \lesssim M^{o(1)}$ we will mean that for every $\epsilon>0$ we have $N \leq C_{\epsilon}M^{\epsilon}$ for some $C_{\epsilon}>0.$

We also make use of big-Oh and little-oh notation: we let $O(Y)$  denote a quantity that is $\lesssim Y$, and similarly
$O_p(Y )$ will denote a quantity that is $\lesssim_p Y$; we let $o_{t \to a}(Y)$
denote a quantity whose quotient with $Y$ tends to zero as $t \to a$ (possibly $\infty$), and
$o_{t \to a;p}(Y)$
denote a quantity whose quotient with $Y$ tends to zero as $t \to a$ at a rate depending on $p$. When clear from context, we will suppress the $t \to a$ subscript.
\subsection{Variation (semi)-norms}
For an ordered set $T \subseteq \mathbb{R}$ and $1 \leq r < \infty$,
we define the homogeneous $r$-variation seminorm by
\[
V^r(a_t : t \in T)
:=
\sup_{\substack{L \geq 1\\
t_0 < \cdots < t_L,\; t_j \in T}}
\left(
\sum_{j=1}^{L}
|a_{t_j}-a_{t_{j-1}}|^r
\right)^{1/r}.
\]
The inhomogeneous $r$-variation norm is defined by
\[
\widetilde{V}^r(a_t : t \in T)
:=
\sup_{t \in T}|a_t| + V^r(a_t : t \in T).
\]
When it is clear from the context, the $t \in T$ notation will be dropped.
\subsection{Gowers Norms and Technology}
\label{sub:gowers_norms}
Let $f \colon \mathbb{Z} \to \mathbb{C} $ be a finitely supported function on the integers. Set the conjugation-difference operator to be 
\begin{align}\label{e:triangle}
  \Delta_h f(x) \coloneqq  f(x) \overline{f(x+h)}, 
  \qquad x,h\in \mathbb{Z} . 
\end{align}
The basic fact here is 
\begin{align} \label{e;double}
\Bigl\lvert \sum_x f(x)\Bigr\rvert ^2 
= 
\sum_{x,h} \Delta_h f(x). 
\end{align}
The higher order conjugation-difference operator is inductively defined to be 
\begin{align}
   \Delta_{h_1,\dots,h_s}f(x) \coloneqq  \Delta_{h_s}(\Delta_{h_1,\dots,h_{s-1}} f)(x), 
    \qquad x,h_1, \ldots, h_s\in \mathbb{Z} . 
\end{align}
Then, the $s$th order Gowers norm is 
\begin{align} \label{e;GowersDef}
    \| f \|_{U^s(\mathbb{Z})}^{2^s} \coloneqq  
    \sum_{x,h_1,\dots,h_s} \Delta_{h_1,\dots,h_s}f(x). 
\end{align}
For $s=1$, this is a semi-norm, while higher orders are norms. In particular, for $s=2$, we have 
\begin{align}  \label{e;U2-ell4}
\lVert f \rVert_{U^2(\mathbb{Z} )} ^{4} 
= \int_{\mathbb{T} } \lvert  \widehat{f} (\beta )\rvert ^{4}\;d \beta;
\end{align}  
we recall the inductive relationship between norms given by 
\begin{align}\label{e:foliate}
\| f \|_{U^{s+1}(\mathbb{Z})}^{2^{s+1}} = \sum_{h_1,\dots,h_{s-1}} \| \Delta_{h_1,\dots,h_{s-1}} f \|_{U^2(\mathbb{Z})}^4.
\end{align}

For integers $N$, let $[N]=\{1,2, \dots, N\}$. 
Denote the usual expectation by 
\begin{align}
 \mathbb{E}_{n \in [N]} f(n) := \frac{1}{N} \sum_{n \in [N]} f(n)\end{align}
and more generally, for finite $X \subset \mathbb{Z}$, define
\begin{align}
 \mathbb{E}_{n \in X} f(n) := \frac{1}{|X|} \sum_{n \in X} f(n) \end{align}
and the norms
\begin{align}
\| f \|_{L^p(X)}^p := \frac{1}{|X|} \sum_{n \in X} |f(n)|^p,
\end{align}
and inner products
\begin{align}\label{e:Innerprod}
    \langle f,g \rangle_X := \frac{1}{|X|} \sum_{n \in X} f(n) \overline{g(n)}.
\end{align}

We then define normalized $U^s$ norms as follows. For $f \colon [N] \to \mathbb C$, 
set 
\begin{equation}
    \lVert f \rVert_{U^s([N]) } \coloneqq  
    \frac{\| f \|_{U^s(\mathbb{Z})}}{\lVert \mathbf{1}_{[N]} \rVert_{U^s(\mathbb Z) }} \approx \frac{\| f \|_{U^s(\mathbb{Z})}}{N^{\frac{s+1}{2^s}}}.
\end{equation}

We make repeated use of the $s=3$ case of the following estimate, due to Eisner-Tao \cite{ET}, relating $U^s$ and $\ell^p$ norms for appropriate $p$ (determined by scaling).
\begin{lemma}\label{l:ET}
    For each $s \geq 1$, set $p_s := \frac{2^s}{s+1}$. Then
    \begin{align}
        \| f \|_{U^s([N])} \lesssim_s \| f \|_{L^{p_s}([N])}.
    \end{align}
\end{lemma}

We will also make use of the following $U^3$ estimate, essentially established in \cite[Proposition 1.14]{F+}. For completeness we include a proof.
\begin{lemma}\label{e:U3maxmod}
The following estimate holds:
\begin{align}
   \left \| \sup_\lambda |\mathbb{E}_{n \in [N]} f(x-n) w(n) e(\lambda n)| \right\|_{L^4([-N,2N])} \lesssim \| w \|_{U^3([N])} \|f \|_{U^3([-3N,3N])},
\end{align}
and in particular
\begin{align}
  \left  \| \sup_\lambda |\mathbb{E}_{n \in [N]} f(x-n) w(n) e(\lambda n)| \right\|_{L^2([-N,2N])} \lesssim \| w \|_{U^3([N])} \|f \|_{L^2([-3N,3N])}.
\end{align}
More generally, whenever $\phi \in \mathcal{C}_c^2([0,1])$ with $\| \phi \|_{\mathcal{C}^2(\mathbb{R})} \leq 1$,
\begin{align}
 \left   \| \sup_\lambda \left|\sum_n \phi_N(n) f(x-n) w(n) e(\lambda n)\right| \right\|_{L^2([-N,2N])} \lesssim \| w \|_{U^3([N])} \|f \|_{L^2([-3N,3N])}.
\end{align}
\end{lemma}
\begin{proof}
By linearization, it suffices to prove the estimate uniformly for an arbitrary selector
\[
    \theta:[-N,2N]\to\mathbb T.
\]
Thus, if we set
\[
    A_\theta f(x)
    :=
    \mathbb E_{n\in[N]} w(n)f(x-n)e(\theta(x)n),
\]
it suffices to prove
\begin{align}\label{e:U3lin-target}
    \|A_\theta f\|_{L^4([-N,2N])}
    \lesssim
    \|w\|_{U^3([N])}
    \|f\|_{U^3([-3N,3N])},
\end{align}
with a constant independent of \(\theta\).

We extend \(w\mathbf 1_{[N]}\) and \(f\mathbf 1_{[-3N,3N]}\) by zero to all of
\(\mathbb Z\), and begin by recording a standard van der Corput identity in normalized form. If \(a\) is
supported on $[N]$, then
\begin{align}\label{e:vdc-local}
    \big|
    \mathbb E_{n \in [N]} a(n)
    \big|^4
    \lesssim
    \mathbb E_{|h|\leq N}
    \big|
    \mathbb E_{n \in [N]} \Delta_h a(n)
    \big|^2.
\end{align}
Indeed,
\[
    \big|\mathbb E_{n \in [N]} a(n)\big|^2
\lesssim
    \mathbb E_{|h|\leq N}   \mathbb E_{n \in [N]} \Delta_h a(n)
\]
up to harmless Fej\'{e}r weights, and \eqref{e:vdc-local} follows by Cauchy--Schwarz in
\(h\).

Apply \eqref{e:vdc-local} for each fixed \(x\) to
\[
    a_x(n):=w(n)f(x-n)e(\theta(x)n);
\]
since
\[
    \Delta_h e(\theta(x)n)
    =
    e(\theta(x)n)\overline{e(\theta(x)(n+h))}
    =
    e(-\theta(x)h),
\]
this phase is independent of \(n\), and therefore disappears after taking absolute values.
Thus
\begin{align}\label{e:first-diff}
    |A_\theta f(x)|^4
    &\lesssim
    \mathbb E_{|h|\leq N}
    \big|
    \mathbb E_{n \in [N]}
    \Delta_h w(n)\,
    \Delta_{-h} f(x-n)
    \big|^2.
\end{align}
Here \(\Delta_{-h}f(y)=f(y)\overline{f(y-h)}\); replacing \(h\) by \(-h\) later is harmless.

Averaging \eqref{e:first-diff} over \(x\in[-N,2N]\), we get
\begin{align}\label{e:xh-average}
    \|A_\theta f\|_{L^4([-N,2N])}^4
    &\lesssim
    \mathbb E_{|h|\leq N}
    \Big\|
    \mathbb E_{n \in [N]}
    \Delta_h w(n)\,
    \Delta_{-h} f(x-n)
    \Big\|_{L^2_x([-N,2N])}^2.
\end{align}

We next use the following \(U^2\)-convolution estimate: for any finitely supported \(F,G\) at
scale \(N\),
\begin{align}\label{e:U2-conv}
    \Big\|
    \mathbb E_{n \in [N]} F(n)G(x-n)
    \Big\|_{L^2_x}
    \lesssim
    \|F\|_{U^2}\|G\|_{U^2},
\end{align}
where the \(U^2\)-norms are normalized at scale \(N\). To see this, write the left-hand
side as a convolution. By Plancherel and Hölder,
\begin{align}
    \|F*G\|_{\ell^2}^2
    &=
    \int_{\mathbb T} |\widehat F(\xi)|^2|\widehat G(\xi)|^2\,d\xi
    \notag\\
    &\leq
    \bigg(\int_{\mathbb T}|\widehat F(\xi)|^4\,d\xi\bigg)^{1/2}
    \bigg(\int_{\mathbb T}|\widehat G(\xi)|^4\,d\xi\bigg)^{1/2}
    \notag\\
    &=
    \|F\|_{U^2(\mathbb Z)}^2\|G\|_{U^2(\mathbb Z)}^2.
\end{align}
After inserting the normalizing factors coming from \(\mathbb E_n\) and
\(L^2_x\), this is precisely \eqref{e:U2-conv}.

Applying \eqref{e:U2-conv} to
\[
    F=\Delta_h w,
    \qquad
    G=\Delta_{-h}f,
\]
we obtain
\begin{align}
    \Big\|
    \mathbb E_{n \in [N]}
    \Delta_h w(n)\,
    \Delta_{-h} f(x-n)
    \Big\|_{L^2_x([-N,2N])}^2
    \lesssim
    \|\Delta_h w\|_{U^2([N])}^2
    \|\Delta_{-h}f\|_{U^2([-3N,3N])}^2.
\end{align}
Substituting this into \eqref{e:xh-average} gives
\begin{align}
    \|A_\theta f\|_{L^4([-N,2N])}^4
    &\lesssim
    \mathbb E_{|h|\leq N}
    \|\Delta_h w\|_{U^2([N])}^2
    \|\Delta_{-h}f\|_{U^2([-3N,3N])}^2.
\end{align}
By Cauchy--Schwarz in \(h\),
\begin{align}
    \|A_\theta f\|_{L^4([-N,2N])}^4
    &\lesssim
    \bigg(
    \mathbb E_{|h|\leq N}
    \|\Delta_h w\|_{U^2([N])}^4
    \bigg)^{1/2}
    \bigg(
    \mathbb E_{|h|\leq N}
    \|\Delta_{-h}f\|_{U^2([-3N,3N])}^4
    \bigg)^{1/2}.
\end{align}
The result follows by the recursive identity for Gowers norms, so we conclude that
\begin{align}
    \|A_\theta f\|_{L^4([-N,2N])}^4
    &\lesssim
    \|w\|_{U^3([N])}^4
    \|f\|_{U^3([-3N,3N])}^4.
\end{align}
Taking fourth roots proves \eqref{e:U3lin-target}, and since the bound is uniform in the linearizing selector \(\theta\), the supremum over \(\lambda\in\mathbb T\) follows.

To conclude the final point, we use Fourier inversion and convexity to bound the left-hand side
\begin{align}
&\int |\widehat{\phi}(\xi)| \cdot \left \| \sup_\lambda |\mathbb{E}_{[N]} f(x-n) (\text{Mod}_{\xi/N} w(n)) e(\lambda n)| \right \|_{L^2([-N,2N])} \ d\xi \\
& \lesssim \| \widehat{\phi} \|_{L^1(\mathbb{R})} \| w \|_{U^3([N])} \|f \|_{L^2([-3N,3N])},
\end{align}
using the invariance of $U^3$ norms under multiplication by characters; the result follows by estimating 
\[ \| \widehat{\phi} \|_{L^1(\mathbb{R})} \lesssim \|\phi \|_{\mathcal{C}^2(\mathbb{R})}.\]
\end{proof}

\subsection{Variational Preliminaries}
We begin with the following cheap inequality from \cite{JSW}: if
\[ J = \bigcup_{n=1}^N J_n, \qquad J_n := J \cap [A_n,B_n] \]
is a disjoint partition, then we may bound
\begin{align}\label{e:split}
\mathcal{V}^r(a_\lambda : \lambda \in J) &\lesssim \left(\sum_n \mathcal{V}^r(a_\lambda : \lambda \in J_n)^r\right)^{1/r} + \mathcal{V}^r(a_{B_n} : n) \\
& \lesssim \left(\sum_n \mathcal{V}^r(a_\lambda : \lambda \in J_n)^r\right)^{1/r} + \left(\sum_{n \leq N} |a_{B_n}|^r\right)^{1/r}.
\end{align}

We next recall the following Lemma, which is by now standard in the field, see e.g.\ \cite{LL} or \cite[\S 4]{BOOK}.
\begin{lemma}\label{l:RMvariation}
Let \((a_j)_{0\leq j\leq 2^J}\) be a finite scalar sequence.  For
\(0\leq m\leq J\), set
\begin{align}
    S_m(a)^2
    :=
    \sum_{s=0}^{2^{J-m}-1}
    |a_{(s+1)2^m}-a_{s2^m}|^2.
\end{align}
Then
\begin{align}
    \mathcal V^2(a_j:0\leq j\leq 2^J)
    \lesssim
    \sum_{m=0}^J S_m(a).
\end{align}
\end{lemma}

We next introduce a $U^3$ estimate to address variation of modulated averages:

\begin{lemma}\label{l:crudeV2lambda}
For every \(N\geq1\), every \(w:[N]\to\mathbb C\), and every finitely supported
\(f:\mathbb Z\to\mathbb C\),
\begin{align}
    \left\|
    \mathcal V^2(A_\lambda f : \lambda \in \mathbb{T})
    \right\|_{\ell^2(\mathbb Z)}
    \lesssim
    \log(2N)\,\|w\|_{L^2([N])}\,\|f\|_{\ell^2(\mathbb Z)},
\end{align}
where
\begin{align}
    A_\lambda f(x)
    :=
    \mathbb E_{n\in[N]} f(x-n)w(n)e(\lambda n).
\end{align}
\end{lemma}

\begin{proof}
We first prove a square-function estimate for dyadic modulation increments.  Let
\(K\geq1\) and \(\alpha\in\mathbb T\); for \(0\leq s<K\), define
\begin{align}
    D_{K,\alpha,s}f
    :=
    A_{\alpha+(s+1)/K}f-A_{\alpha+s/K}f.
\end{align}
We claim that
\begin{align}\label{e:Dsfbound}
    \left\|
    \left(\sum_{s=0}^{K-1}|D_{K,\alpha,s}f|^2\right)^{1/2}
    \right\|_{\ell^2}
    \lesssim
    \min\left\{1,(N/K)^{1/2}\right\}
    \|w\|_{L^2([N])}\|f\|_{\ell^2},
\end{align}
with an absolute implicit constant, uniformly in \(K\) and \(\alpha\).

If we define
\begin{align}
    m(\xi):=\frac1N\sum_{n=1}^N w(n)e(-\xi n),
\end{align}
then \(A_\lambda\) has Fourier multiplier \(m(\xi-\lambda)\), while the multiplier of
\(D_{K,\alpha,s}\) is
\begin{align}
    &m(\xi-\alpha-(s+1)/K)-m(\xi-\alpha-s/K) \\
    & \qquad = \frac1N\sum_{n=1}^N
    w(n)e((\alpha-\xi)n)e(sn/K)(e(n/K)-1),
\end{align}
so Plancherel and orthogonality \(\mod K\) allow us to express
\begin{align}
    &\sum_{s=0}^{K-1}
    \left|
    m(\xi-\alpha-(s+1)/K)-m(\xi-\alpha-s/K)
    \right|^2\notag\\
    &\qquad=
    K\sum_{l \in\mathbb Z/K\mathbb Z}
    \big|
    \frac{1}{N} \sum_{\substack{1\leq n\leq N\\ n\equiv l \mod K}}
    w(n) e((\alpha-\xi)n)(e(n/K)-1)
    \big|^2.
\end{align}
If \(K\leq N\), then each residue class modulo \(K\) contains at most \(2N/K\)
integers in \([N]\), and 
\[ |e(n/K)-1|\leq2;\]  Cauchy-Schwarz inside each residue class gives
\begin{align}
    \sum_{s=0}^{K-1}
    \left|
    m(\xi-\alpha-(s+1)/K)-m(\xi-\alpha-s/K)
    \right|^2
    \lesssim
    \frac1N\sum_{n=1}^N|w(n)|^2.
\end{align}
If \(K>N\), each residue class contains at most one integer in \([N]\).  Since
\[ |e(n/K)-1|\lesssim  N/K,\]
we get
\begin{align}
    \sum_{s=0}^{K-1}
    \left|
    m(\xi-\alpha-(s+1)/K)-m(\xi-\alpha-s/K)
    \right|^2
    \lesssim
    \frac NK\cdot \frac1N\sum_{n=1}^N|w(n)|^2.
\end{align}
The two cases prove the pointwise multiplier square-function bound
\begin{align}\label{e:multsf}
    &\sup_{\xi\in\mathbb T}
    \left(
    \sum_{s=0}^{K-1}
    \left|
    m(\xi-\alpha-(s+1)/K)-m(\xi-\alpha-s/K)
    \right|^2
    \right)^{1/2} \\
& \qquad     \lesssim
    \min\left\{1,(N/K)^{1/2}\right\}
    \|w\|_{L^2([N])},
\end{align}
from which \eqref{e:Dsfbound} follows by Plancherel.

We now prove the variation estimate.  First restrict \(\lambda\) to the dyadic grid
\[ \{\frac{j}{2^J} :0\leq j\leq2^J\}\]
and apply Lemma \ref{l:RMvariation} pointwise in
\(x\) to the sequence
\begin{align}
    a_j(x):=A_{j2^{-J}}f(x);
\end{align}
taking \(\ell^2_x\)-norms yields
\begin{align}
    \left\|\mathcal V^2_j(A_{j2^{-J}}f:0\leq j\leq2^J)\right\|_{\ell^2_x}
    &\lesssim
    \sum_{m=0}^J
    \left\|
    \left(
    \sum_{s=0}^{2^{J-m}-1}
    |A_{(s+1)/2^{J-m}}f-A_{s/2^{J-m}}f|^2
    \right)^{1/2}
    \right\|_{\ell^2_x}.
\end{align}
Putting \(K=2^{J-m}\) and using \eqref{e:Dsfbound} with \(\alpha=0\) yields the estimate
\begin{align}
    \left\|\mathcal V^2_j(A_{j2^{-J}}f:0\leq j\leq2^J)\right\|_{\ell^2}
    &\lesssim
    \|w\|_{L^2([N])}\|f\|_{\ell^2}
    \sum_{\substack{K=2^\ell\\0\leq \ell\leq J}}
    \min\left\{1,(N/K)^{1/2}\right\}\\
    &\lesssim
    \log(2N)\,\|w\|_{L^2([N])}\|f\|_{\ell^2}.
\end{align}
Since this estimate is uniform in \(J\) and \(f\) is finitely supported, \(\lambda\mapsto
A_\lambda f(x)\) is continuous for every \(x\), so the result follows by a limiting argument.
\end{proof}
The following corollary follows immediately by interpolation.
\begin{cor}[$U^3$-control of modulation variation]\label{c:U3modvariation}
The following bound holds for each $r \geq 2$:
\begin{align}\label{e:U3varlambda}
    \left\|\mathcal V^r(A_\lambda f : \lambda \in \mathbb{T}) \right\|_{\ell^2(\mathbb Z)}
    \lesssim_r
    (\log(2N))^{2/r}
    \|w\|_{U^3([N])}^{1-2/r}
    \|w\|_{L^2([N])}^{2/r}
    \|f\|_{\ell^2(\mathbb Z)}.
\end{align}
\end{cor}

Our next lemma serves as the $\ell^p$ counterpart to Lemma \ref{l:crudeV2lambda}.

\begin{lemma}
\label{l:crude-lambda-var-scale}
Let \(r>2\), and let \(r':=r/(r-1)\).  Let \(N_0\geq2\), and suppose that $w$ is supported in \(\{|n|\leq N_0\}\). If we set
\begin{align}
    T_\lambda f(x) := T_\lambda^w f(x) := \sum_n f(x-n) w(n) e(\lambda n),
\end{align}
then, for every \(p\geq r'\),
\[
    \left\|
    \mathcal V^r (T_\lambda f : \lambda \in \mathbb{T})
    \right\|_{\ell^p(\mathbb Z)}
    \lesssim_{p,r}
    N_0^{1/r}\|w\|_{\ell^{r'}(\mathbb Z)}
    \|f\|_{\ell^p(\mathbb Z)},
\]
while for \(1\leq p<r'\),
\[
    \left\|
    \mathcal V^r(T_\lambda f : \lambda \in \mathbb{T})
    \right\|_{\ell^p(\mathbb Z)}
    \lesssim_{p,r}
    N_0^{1/r}\|w\|_{\ell^{p}(\mathbb Z)}
    \|f\|_{\ell^p(\mathbb Z)}.
\]
\end{lemma}

\begin{proof}
We first record a scalar estimate: if
\[
    P(\lambda):=\sum_{|n|\leq N_0} c_n e(\lambda n),
\]
then we may bound
\begin{align}\label{e:embed}
    \mathcal V^r(P(\lambda) : \lambda \in \mathbb{T})
    \lesssim_r
    N_0^{1/r}
    \left(\sum_{|n|\leq N_0}|c_n|^{r'}\right)^{1/r'}.
\end{align}
To see this, partition \([0,1]\) into intervals
\[
    I_j:=[\frac{j}{N_0},\frac{j+1}{N_0}),
    \qquad 0\leq j<N_0,
\]
and apply \eqref{e:split}: since \(r>1\), we may apply H\"{o}lder and bound
\[
    \mathcal V^r(P(\lambda): \lambda \in I_j)\leq \int_{I_j}|P'(\lambda)|\,d\lambda \leq
    |I_j|^{1/r'}
    \big(\int_{I_j}|P'(\lambda)|^r\,d\lambda\big)^{1/r},
\]
so 
\[
    \left(
    \sum_j \mathcal V^r(P(\lambda): \lambda \in I_j)^r
    \right)^{1/r}
    \leq
    N_0^{-1/r'}\|P'\|_{L^r(\mathbb T)} \lesssim N_0^{1/r} \| c \|_{\ell^{r'}}
\]
by Hausdorff-Young; we similarly apply the discrete Hausdorff-Young inequality on \(\mathbb Z/N_0\) -- after separating the
finitely many residue classes of the frequencies \(|n|\leq N_0\) -- to address the endpoint term,
\[
    \left(\sum_{j=0}^{N_0-1}|P(j/N_0)|^r\right)^{1/r}
    \lesssim_r
    N_0^{1/r}
    \|c\|_{\ell^{r'}},
\]
completing the reduction.

We now apply \eqref{e:embed} pointwise in \(x\) with
\[
    c_n=w(n)f(x-n),
\]
obtaining
\[
    \mathcal V^r(T_\lambda f(x) : \lambda \in \mathbb{T})
    \lesssim_r
    N_0^{1/r}
    \left(
    \sum_n |w(n)|^{r'}|f(x-n)|^{r'}
    \right)^{1/r'};
\]
if \(p\geq r'\), then \(p/r'\geq1\), so Minkowski gives
\begin{align}
    \left\|
    \left(
    \sum_n |w(n)|^{r'}|f(x-n)|^{r'}
    \right)^{1/r'}
    \right\|_{\ell^p_x}
    &=
    \left\|
    \sum_n |w(n)|^{r'}|f(x-n)|^{r'}
    \right\|_{\ell^{p/r'}_x}^{1/r'}  \\
    &\leq
    \left(
    \sum_n |w(n)|^{r'}\|f(\cdot-n)\|_{\ell^p}^{r'}
    \right)^{1/r'}  \\
    &=
    \|w\|_{\ell^{r'}}\|f\|_{\ell^p},
\end{align}
proving the first estimate. In the opposite case, if \(1\leq p<r'\), then by the subadditivity of \(t\mapsto t^{p/r'}\), we bound
\begin{align}
    \left\|
    \left(
    \sum_n |w(n)|^{r'}|f(x-n)|^{r'}
    \right)^{1/r'}
    \right\|_{\ell^p_x}^p
    &\leq
    \sum_x\sum_n |w(n)|^p|f(x-n)|^p   \\
    &=
    \|w\|_{\ell^p}^p\|f\|_{\ell^p}^p,
\end{align}
proving the second estimate.
\end{proof}

With the above in mind, we begin our arguments proper.

\section{Reductions}\label{s:red}
Fix \(r>2\); we here reduce Theorem \ref{t:main} to establishing more arithmetically tractable estimate. To do so, for each $1 < p < \infty$, we choose a constant $D_0 := D_0(p,r)$ sufficiently large, and then choose $0 < \epsilon \ll_{p,r} 1$ sufficiently small; in particular, in the periodic section we will choose
\[
    \kappa=1/D_0-\epsilon,
\]
see \eqref{e:kappa} below.  We set
\begin{align}\label{e:KQ}
    K_Q(n):=K_{Q,D_0}(n)
    :=
    \varphi_{N(Q)}*
    \sum_{Q^{1/D_0}\leq k}\psi_k(n),
\end{align}
where \(\varphi\) is a Schwartz function whose Fourier transform is supported
in a sufficiently small fixed neighborhood of the origin, and
\[
    N(Q):=\exp(Q^{1/D_0-\epsilon/10}).
\]
With this definition in mind, define the \(Q\)-denominator variational block
\begin{align}\label{e:CQvar}
    \mathcal V^{Q,r}f(x)
    :=
    \mathcal V^r
    \left(
    \sum_n f(x-n)w_Q(n)K_Q(n)e(\lambda n) : \lambda \in \mathbb{T}
    \right); 
\end{align}
in the limiting $r \to \infty$ case, we define
\begin{align}
    \mathcal C^{Q}f(x)
    :=
    \sup_{\lambda \in \mathbb{T}}
    \left|
    \sum_n f(x-n)w_Q(n)K_Q(n)e(\lambda n)\right|.
\end{align}

Our first order of business is to reduce to estimates for the operators
\(\mathcal V^{Q,r}\).

\begin{proposition}\label{p:red-var}
Fix \(r>2\), and set \(r':=r/(r-1)\).  Let \(p>r'\).  To prove the
\(r\)-variational analogue
\[
        \|\mathcal V_{\mathbb P}^{r}f\|_{\ell^p}
        \lesssim_{p,r}
        \|f\|_{\ell^p},
\]
it suffices to prove that for each dyadic \(Q\geq1\) there exists \(c_{p,r}>0\)
such that
\begin{align}\label{e:red-var}
        \|\mathcal V^{Q,r}f\|_{\ell^p}
        \lesssim_{p,r}
        Q^{-c_{p,r}}\|f\|_{\ell^p}.
\end{align}
\end{proposition}

\begin{proof}
Since
\[
    \sum_{k\geq1}\psi_k(n)=\frac1n,
    \qquad |n|\geq1,
\]
we have
\[
    \mathcal V_{\mathbb P}^{r}f(x)
    =
    \mathcal{V}^r\left(
    \sum_n f(x-n)\sum_{k\geq 1} w(n)\psi_k(n)e(\lambda n):
    \lambda\in\mathbb T
    \right),
\]
so by the triangle inequality for \(r\)-variation,
\begin{align}
    \mathcal V_{\mathbb P}^{r}f(x)
    &\leq
    \mathcal{V}^r\left(
    \sum_n f(x-n)\sum_{k \geq 1} w_{\leq M(k)}(n)\psi_k(n)e(\lambda n):
    \lambda\in\mathbb T
    \right)
    \notag\\
    &\quad+
    \sum_{k\geq1}
    \mathcal{V}^r\left(
    \sum_n f(x-n)(w-w_{\leq M(k)})(n)\psi_k(n)e(\lambda n):
    \lambda\in\mathbb T
    \right),
\end{align}
where
\[
    M(k):=\exp(k^{1/10}),
    \qquad
    w_{\leq M(k)}:=\sum_{L\leq M(k)}w_L,
\]
with \(L\) dyadic.

To estimate the residual term, use \cite[Proposition 3.1]{F+} to bound
\begin{align}\label{e:minor-U3-var-red}
    \|w-w_{\leq M(k)}\|_{U^3([2^k])}
    \lesssim_A k^{-A}
\end{align}
for every \(A<\infty\), and then apply the smooth-cutoff form of Corollary
\ref{c:U3modvariation} at scale \(2^k\) to bound
\begin{align}\label{e:local-U3-var-red}
    &\left\|
    \mathcal{V}^r\left(
    \sum_n f(x-n)u(n)\psi_k(n)e(\lambda n):\lambda\in\mathbb T
    \right)
    \right\|_{\ell^2}
    \notag\\
    &\qquad\lesssim_{r}
    k^{2/r}
    \|u\|_{U^3([2^k])}^{1-2/r}
    \|u\|_{L^2([2^k])}^{2/r}
    \|f\|_{\ell^2},
\end{align}
by \eqref{e:U3varlambda} and convexity.

Taking \(u=w-w_{\leq M(k)}\), using \eqref{e:minor-U3-var-red}, and using the crude local moment bound
\[
    \|w-w_{\leq M(k)}\|_{L^2([-2^k,2^k])}
    \lesssim k^{O(1)},
\]
see \cite[Lemma 3.1]{FKT1}, we obtain
\begin{align}\label{e:minor-var-L2-red}
    &\left\|
    \mathcal{V}^r\left(
    \sum_n f(x-n)(w-w_{\leq M(k)})(n)\psi_k(n)e(\lambda n):
    \lambda\in\mathbb T
    \right)
    \right\|_{\ell^2}
    \notag\\
    &\qquad\lesssim_{A,r}
    k^{-A(1-2/r)+O_r(1)}
    \|f\|_{\ell^2}.
\end{align}

We next record the crude bound at the auxiliary exponent \(p > s \geq r'\). Abbreviate
\[
    u_k(n):=(w-w_{\leq M(k)})(n);
\]
by Lemma \ref{l:crude-lambda-var-scale}, with \(N_0\approx 2^k\), 
\begin{align}
    &\left\|
    \mathcal{V}^r\left(
    \sum_n f(x-n)u_k(n)\psi_k(n)e(\lambda n):\lambda\in\mathbb T
    \right)
    \right\|_{\ell^s}
    \notag\\
    &\qquad\lesssim_{s,r}
    2^{k/r}\|u_k\psi_k\|_{\ell^{r'}}\|f\|_{\ell^s} \\
    & \qquad \lesssim \big(
    \mathbb E_{|n|\approx 2^k}|u_k(n)|^{r'}
    \big)^{1/r'} \|f\|_{\ell^s}.
\end{align}
Since \(|w(n)|\lesssim k\) on \(|n|\approx2^k\), and
by Minkowski's inequality and \cite[Lemma 3.1]{FKT1},
\begin{align}
    \big(
    \mathbb E_{|n|\approx 2^k}|w_{\leq M(k)}(n)|^{r'}
    \big)^{1/r'}
    &\leq
    \sum_{L\leq M(k)}
    \big(
    \mathbb E_{|n|\approx 2^k}|w_L(n)|^{r'}
    \big)^{1/r'}
    \notag\\
    &\lesssim_r
    \sum_{L\leq M(k)}(\log L)^{O_r(1)}
    \notag\\
    &\lesssim_r
    k^{O_r(1)},
\end{align}
we conclude
\begin{align}\label{e:minor-var-crude-red}
    &\Big\|
    \mathcal{V}^r\big(
    \sum_n f(x-n)(w-w_{\leq M(k)})(n)\psi_k(n)e(\lambda n):
    \lambda\in\mathbb T
    \big)
    \Big\|_{\ell^s}
    \notag\\
    &\qquad\lesssim_{s,r}
    k^{O_r(1)}
    \|f\|_{\ell^s}.
\end{align}

Interpolating \eqref{e:minor-var-L2-red} and \eqref{e:minor-var-crude-red} and choosing $A = A(p,r)$ sufficiently large allows us to sum
\begin{align}
    &
    \sum_{k\geq1} \left\|
    \mathcal{V}^r\left(
    \sum_n f(x-n)(w-w_{\leq M(k)})(n)\psi_k(n)e(\lambda n):
    \lambda\in\mathbb T
    \right)
    \right\|_{\ell^p}
    \notag\\
    &\qquad\lesssim_{p,r}
    \sum_{k\geq1}k^{-2}\|f\|_{\ell^p}
    \lesssim_{p,r}
    \|f\|_{\ell^p}.
\end{align}

It remains to treat the main term. Again by the triangle inequality for \(r\)-variation,
\begin{align}
    &\mathcal{V}^r\left(
    \sum_n f(x-n)\sum_{k \geq 1} w_{\leq M(k)}(n)\psi_k(n)e(\lambda n):
    \lambda\in\mathbb T
    \right)
    \notag\\
    &\leq
    \mathcal{V}^r\left(
    \sum_n f(x-n)\sum_{k \geq 1}\sum_{Q\leq k^{D_0}}
    w_Q(n)\psi_k(n)e(\lambda n):
    \lambda\in\mathbb T
    \right)
    \notag\\
    &\quad+
    \sum_{k \geq 1}\sum_{k^{D_0}\leq Q\leq M(k)}
    \mathcal{V}^r\left(
    \sum_n f(x-n)w_Q(n)\psi_k(n)e(\lambda n):
    \lambda\in\mathbb T
    \right),
\end{align}
where \(Q\) ranges over dyadic integers. To estimate the high-denominator term, we use \eqref{e:local-U3-var-red} to interpolate between \cite[Proposition 2.1]{F+},
\[
    \|w_Q\|_{U^3([2^k])}
    \lesssim Q^{-3/8+o(1)},
\]
and \cite[Lemma 3.1]{FKT1},
\[
    \|w_Q\|_{L^2([2^k])}\lesssim Q^{o(1)},
\]
yielding
\begin{align}\label{e:highQ-var-L2-red}
    &\left\|
    \mathcal{V}^r\left(
    \sum_n f(x-n)w_Q(n)\psi_k(n)e(\lambda n):
    \lambda\in\mathbb T
    \right)
    \right\|_{\ell^2}
    \notag\\
    &\qquad\lesssim_r
    k^{2/r}
    Q^{-\frac38(1-2/r)+o(1)}
    \|f\|_{\ell^2}.
\end{align}
A crude estimate at exponent \(s \geq r'\) follows again from Lemma
\ref{l:crude-lambda-var-scale}:
\begin{align}
    &\left\|
    \mathcal{V}^r\left(
    \sum_n f(x-n)w_Q(n)\psi_k(n)e(\lambda n):
    \lambda\in\mathbb T
    \right)
    \right\|_{\ell^s}
    \notag\\
    &\qquad\lesssim_{s,r}
    2^{k/r}\|w_Q\psi_k\|_{\ell^{r'}}\|f\|_{\ell^s}
    \notag\\
    &\qquad\lesssim_{s,r}
    \big(
    \mathbb E_{|n|\approx 2^k}|w_Q(n)|^{r'}
    \big)^{1/r'}
    \|f\|_{\ell^s}
    \notag\\
    &\qquad\lesssim_{s,r}
    Q^{o(1)}\|f\|_{\ell^s},
\end{align}
where the final line is another application of \cite[Lemma 3.1]{FKT1}.  Interpolating
with \eqref{e:highQ-var-L2-red} allows us to sum over $Q$ satisfying
\[ k^{D_0} \leq Q \leq M(k).\]

We are left with the low-denominator contribution, to which we again apply the triangle inequality,
\begin{align}
    &\mathcal{V}^r\left(
    \sum_n f(x-n)\sum_{k \geq 1}\sum_{Q\leq k^{D_0}}
    w_Q(n)\psi_k(n)e(\lambda n):
    \lambda\in\mathbb T
    \right)
    \notag\\
    &\leq
    \sum_Q
    \mathcal{V}^r\left(
    \sum_n f(x-n)w_Q(n)
    \sum_{Q^{1/D_0}\leq k}\psi_k(n)e(\lambda n):
    \lambda\in\mathbb T
    \right).
\end{align}
For each fixed \(Q\), we replace
\[
    \sum_{Q^{1/D_0}\leq k}\psi_k
\]
by \(K_Q\).  The difference has Fourier multiplier of total variation \(O_A(Q^{-A})\)
for every \(A<\infty\), by the definition of \(N(Q)=\exp(Q^{1/D_0-\epsilon/10})\)
and the rapid decay of the Fourier transforms of the kernels \(\psi_k\) at frequencies
larger than \(N(Q)^{-1}\).  Expanding \(w_Q\) into its rational Fourier modes, applying
Lemma \ref{l:varcarZ} to each shifted error multiplier, and using
\[
    \sum_{a/q\in\Gamma_Q}\frac1{\phi(q)}
    \lesssim Q^{1+o(1)}
\]
yields the bound
\begin{align} \label{e:smoothing}
    &\left\|
    \mathcal{V}^r\left(
    \sum_n f(x-n)w_Q(n) \cdot 
    \left(
    \sum_{Q^{1/D_0}\leq k}\psi_k(n)-K_Q(n)
    \right)
    e(\lambda n):\lambda\in\mathbb T
    \right)
    \right\|_{\ell^2}
    \notag\\
    &\qquad\lesssim_{r,A}
    Q^{-A}\|f\|_{\ell^2};
\end{align}
the same crude estimate used above gives a polynomial-in-\(Q\) bound on \(\ell^s\), and interpolating yields
\[
    O_{p,r}(Q^{-100}\|f\|_{\ell^p}).
\]
Combining estimates
we obtain
\begin{align}
    \|\mathcal V_{\mathbb P}^{r}f\|_{\ell^p}
    \lesssim_{p,r}
    \|f\|_{\ell^p}
    +
    \sum_Q\|\mathcal V^{Q,r}f\|_{\ell^p};
\end{align}
the hypothesis \eqref{e:red-var} makes the final dyadic \(Q\)-sum convergent, and the
proposition follows.
\end{proof}

The preceding reduction also has a maximal endpoint: with \(\mathcal{C}_{\mathbb{P}}\) as in \eqref{e:CP}, to prove Theorem \ref{t:carmax} it suffices to prove that for each dyadic \(Q\ge 1\) there exists \(c_p>0\) such that
\[
\| \mathcal{C}^Q f\|_{\ell^p}\lesssim_p Q^{-c_p}\|f\|_{\ell^p}. 
\]
Indeed, the proof of Proposition \ref{p:red-var} carries over with every \( \mathcal{V}^r\)-norm replaced by
a supremum over \(\lambda\). The residual and high-denominator terms are controlled
by the maximal \(U^3\)-estimate of Lemma \ref{e:U3maxmod} and the same \(U^3\)-decay inputs,
while the replacement of \(\sum_{Q^{1/D_0}\le k}\psi_k\) by \(K_Q\) is identical; the only remaining contribution is the dyadic sum of the operators \( \mathcal{C}^Q\).

We next record the extent to which the variational estimates reduce to an
\(\ell^2\)-based estimate; while this reduction allows one to recover the complete range of $\ell^p$ estimates for $1 < p < \infty$ for the maximal formulation of the Carleson operator, the range of exponents in the variational formulation is somewhat restricted, and dependent on interpolation. We describe the process now:

For each $p > r', \ p \neq 2$, we choose $s := s(p) > r'$ so that
\begin{align}
\begin{cases}
    s < p & \text{ if } r' < p <2 \\
    s > p & \text{ otherwise} 
\end{cases}
\end{align}
and define $0 < \theta = \theta(p) < 1$ by
\begin{align}
    \frac{1}{p} = \frac{\theta}{2} + \frac{1-\theta}{s}.
\end{align}

With this in mind, the reduction is as follows.
\begin{lemma}\label{l:red-var-l2}
Fix \(r>2\), and suppose that, for some \(\sigma>0\),
\begin{align}\label{e:red-var-l2-hyp}
    \|\mathcal V^{Q,r}f\|_{\ell^2}
    \lesssim_r
    Q^{-\sigma}\|f\|_{\ell^2}.
\end{align}
If
\begin{align}\label{e:interp-var-l2-condition}
    \theta\sigma>\frac{1-\theta}{r},
\end{align}
then there exists 
\[ 0 < c_{p,r} < \theta\sigma-\frac{1-\theta}{r} \]
such that
\begin{align}\label{e:red-var-l2-conclusion}
    \|\mathcal V^{Q,r}f\|_{\ell^p}
    \lesssim_{p,r}
    Q^{-c_{p,r}}\|f\|_{\ell^p}.
\end{align}

In the corresponding maximal formulation, if
\[
\| \mathcal{C}^Q f\|_{\ell^2}\lesssim Q^{-\sigma}\|f\|_{\ell^2}
\]
for some \(\sigma>0\), then for every \(1<p<\infty\) there exists \(c_p>0\) such that
\[
\| \mathcal{C}^Q f\|_{\ell^p}\lesssim_p Q^{-c_p}\|f\|_{\ell^p}.
\]
\end{lemma}

To prove Lemma \ref{l:red-var-l2} we will use the elementary estimate
\begin{align}\label{e:var-covering-loss}
    \mathcal{V}^r\big(a_{d\lambda}:\lambda\in\mathbb T\big)
    \lesssim_r
    d^{1/r}\,
    \mathcal{V}^r\big(a_\lambda:\lambda\in\mathbb T\big)
    +
    \left(\sum_{0\leq i<d}|a_{i/d}|^r\right)^{1/r}
\end{align}
for every \(1\)-periodic scalar family \((a_\lambda)_{\lambda\in\mathbb T}\); the limiting case $r = \infty$ is agnostic to dilations.

\begin{proof}
If \(p=2\), the conclusion is exactly \eqref{e:red-var-l2-hyp}.  Hence assume
\(p\neq2\).  By interpolation between \eqref{e:red-var-l2-hyp} and a crude
\(\ell^s\)-estimate, it suffices to prove
\begin{align}\label{e:crude-var-Q-bound}
    \|\mathcal V^{Q,r}f\|_{\ell^s}
    \lesssim_{s,r}
    Q^{1/r+o(1)}\|f\|_{\ell^s}
\end{align}
for every \(s>r'\).

We prove \eqref{e:crude-var-Q-bound} by M\"obius inversion, recalling that
\[
    w_Q(n)
    =
    \sum_{Q/2<q\leq Q}
    \frac{\mu(q)}{\phi(q)}c_q(n),
    \qquad
    c_q(n)
    =
    \sum_{d\mid q} d\,\mu(q/d)\mathbf 1_{d\mid n}.
\]
Thus, by the triangle inequality for \(r\)-variation,
\begin{align}
    \mathcal V^{Q,r}f(x)
    &\leq
    \sum_{Q/2<q\leq Q}
    \frac1{\phi(q)}
    \sum_{d\mid q}
    \mathcal{V}^r\left(
    d\sum_m f(x-dm)K_Q(dm)e(\lambda dm):\lambda\in\mathbb T
    \right).
\end{align}
If we define
\[
    K_{Q,d}(m):=dK_Q(dm),
\]
then the inner variation is
\[
    \mathcal{V}^r\left(
    \sum_m f(x-dm)K_{Q,d}(m)e(d\lambda m):\lambda\in\mathbb T
    \right).
\]
Applying \eqref{e:var-covering-loss} to the \(1\)-periodic family
\[
    a_\lambda(x)
    :=
    \sum_m f(x-dm)K_{Q,d}(m)e(\lambda m),
\]
we obtain
\begin{align}
    &\mathcal{V}^r\left(
    \sum_m f(x-dm)K_{Q,d}(m)e(d\lambda m):\lambda\in\mathbb T
    \right)
    \notag\\
    &\qquad\lesssim_r
    d^{1/r}
    \mathcal{V}^r\left(
    \sum_m f(x-dm)K_{Q,d}(m)e(\lambda m):\lambda\in\mathbb T
    \right)
    \notag\\
    &\qquad\quad+
    \left(
    \sum_{0\leq i<d}
    \left|
    \sum_m f(x-dm)K_{Q,d}(m)e(im/d)
    \right|^r
    \right)^{1/r}.
\end{align}
The sum is harmless, since pointwise
\begin{align}
    &\left(
    \sum_{0\leq i<d}
    \left|
    \sum_m f(x-dm)K_{Q,d}(m)e(im/d)
    \right|^r
    \right)^{1/r}
    \notag\\
    &\qquad\leq
    d^{1/r}
    \sup_{\lambda\in\mathbb T}
    \left|
    \sum_m f(x-dm)K_{Q,d}(m)e(\lambda m)
    \right|.
\end{align}
Therefore
\begin{align}\label{e:mobius-var-pointwise}
    \mathcal V^{Q,r}f(x)
    &\lesssim_r
    \sum_{Q/2<q\leq Q}
    \frac1{\phi(q)}
    \sum_{d\mid q}d^{1/r}
    \mathcal{V}^r\left(
    \sum_m f(x-dm)K_{Q,d}(m)e(\lambda m):\lambda\in\mathbb T
    \right)
    \notag\\
    &\quad+
    \sum_{Q/2<q\leq Q}
    \frac1{\phi(q)}
    \sum_{d\mid q}d^{1/r}
    \sup_{\lambda\in\mathbb T}
    \left|
    \sum_m f(x-dm)K_{Q,d}(m)e(\lambda m)
    \right|.
\end{align}

The kernels \(K_{Q,d}\) have $\| \widehat{K}_{Q,d} \|_{\text{BV}} \lesssim 1$ uniformly in
\(Q\) and \(d\leq Q\), so Corollary \ref{c:disc} yields
\begin{align}
    \left\|
    \mathcal{V}^r\left(
    \sum_m f(x-dm)K_{Q,d}(m)e(\lambda m):\lambda\in\mathbb T
    \right)
    \right\|_{\ell^s_x}
    \lesssim_{s,r}
    \|f\|_{\ell^s},
    \qquad s>r';
\end{align}
the supremum term in \eqref{e:mobius-var-pointwise} is bounded in the same way. Taking \(\ell^s\)-norms in \eqref{e:mobius-var-pointwise}, we obtain
\begin{align}
    \|\mathcal V^{Q,r}f\|_{\ell^s}
    &\lesssim_{s,r}
    \sum_{Q/2<q\leq Q}
    \frac1{\phi(q)}
    \sum_{d\mid q}d^{1/r}\,
    \|f\|_{\ell^s} \lesssim Q^{1/r + o(1)} \|f \|_{\ell^s};
\end{align}
this proves \eqref{e:crude-var-Q-bound}.

Interpolating \eqref{e:red-var-l2-hyp} with \eqref{e:crude-var-Q-bound} yields the bound
\begin{align}
    \|\mathcal V^{Q,r}f\|_{\ell^p}
    &\lesssim_{p,r}
    Q^{-\theta\sigma+(1-\theta)/r+o(1)}
    \|f\|_{\ell^p},
\end{align}
where the exponent is negative by \eqref{e:interp-var-l2-condition}; we obtain
\eqref{e:red-var-l2-conclusion} for any
\[
    0<c_{p,r}<\theta\sigma-\frac{1-\theta}{r}.
\]
In the non-variational supremum formulation, the change of variables
\(\lambda\mapsto d\lambda\) costs nothing, and the analogue of
\eqref{e:crude-var-Q-bound} is \(Q^{o(1)}\), so an arbitrary power saving at
\(\ell^2\) interpolates to all \(1<p<\infty\).
\end{proof}

Lemma \ref{l:red-var-l2} encodes the final interpolation step in the proof. We will obtain the relevant $\ell^2$ power savings for $\mathcal{V}^{Q,r}$ from two estimates. First, 
Proposition \ref{p:global-shell-var} gives, for every $\rho>2$,
\[
    \|\mathcal V^{Q,\rho} f\|_{\ell^2}
    \lesssim
    Q^{o(1)}\|f\|_{\ell^2};
\]
second, \S \ref{s:per}-\ref{s:number} prove the maximal saving
\[
    \|\mathcal C^Q f\|_{\ell^2}
    \lesssim Q^{-\epsilon/2}\|f\|_{\ell^2}.
\]
Fix \(0<\alpha<1-2/r\), and set \(\rho=r(1-\alpha)>2\); applying scalar interpolation we obtain the pointwise bound
\[
    \mathcal V^{Q,r}f
    \lesssim
    \big(\mathcal V^{Q,\rho}f\big)^{1-\alpha}
    \big(\mathcal C^Q f\big)^\alpha.
\]
which yields the estimate
\[
    \|\mathcal V^{Q,r} f\|_{\ell^2}
    \lesssim_{r,\alpha}
    Q^{o(1)-\epsilon\alpha/2}\|f\|_{\ell^2}
\]
by H\"{o}lder's inequality; Lemma \ref{l:red-var-l2} then establishes 
the desired \(\ell^p\) estimates in the range allowed by \eqref{e:interp-var-l2-condition}.

With this plan in mind, we continue to the analytic core of the paper.
\section{Multi-Frequency Analysis}\label{s:MFanalysis}

\subsection{Variational Estimates}
We begin this section by developing some variational estimates for maximally modulated singular integral operators.

\begin{lemma}\label{l:varcarZ}
The following estimate holds whenever $\widehat{k}$ is supported inside an interval of length $\leq 1/10$ and has total variation $\leq c_0$.
    \begin{align}
     \left   \| \mathcal{V}^r \left( \sum_{n} f(x-n) e(\lambda n) k(n) : \lambda\right) \right \|_{\ell^2(\mathbb{Z})} \lesssim_r  c_0 \|f \|_{\ell^2(\mathbb{Z})}. 
    \end{align}
\end{lemma}
To prove this lemma, our departure point is the following deep result of Oberlin et.\ al.\ \cite{O+}:
\begin{align}\label{e:O+}
    \| \mathcal{V}^r( (\widehat{f}(\xi) \mathbf{1}_{(-\infty,0]}(\xi - \lambda) )^{\vee} : \lambda) \|_{L^2(\mathbb{R})} \lesssim_r \|f \|_{L^2(\mathbb{R})},
\end{align}
where the dependence on $r$ follows from inspecting the proof, noting the application of L\'{e}pingle's inequality \cite{LE}.

The estimate \eqref{e:O+} is ``universal," in that it leads to a range of similar estimates via convexity. Specifically, we have the following lemma.
\begin{lemma}\label{l:varcar}
    Suppose that $dm$ is a measure with distribution function $m(t)$. Then
    \begin{align}
        \mathcal{V}^r\big((\widehat{f}(\xi) m(\xi - \lambda) )^{\vee} : \lambda \big) \lesssim \| dm \|_{\text{TV}(\mathbb{R})} \cdot
        \mathcal{V}^r( (\widehat{f}(\xi) \mathbf{1}_{(-\infty,0]}(\xi - \lambda) )^{\vee} : \lambda) 
    \end{align}
\end{lemma}
\begin{proof}
By the fundamental theorem of calculus, we may express
\begin{align}
    \Big( \widehat{f}(\xi) \cdot \big( m(\xi - \lambda_i)  - m(\xi - \lambda_{i+1}) \big) \Big)^{\vee} = \int_{\mathbb{R}} \big( \mathbf{1}_{(t+\lambda_i,t+\lambda_{i+1}]} \cdot \widehat{f}(\xi) \big)^{\vee}(x)  \ dm(t);
\end{align}
if we let $\{ a_i \}$ be an arbitrary, finite sequence of scalars, normalized so that 
\[ \sum_i |a_i|^{r'} \leq 1,\]
then we may bound
\begin{align}
    &\Big| \sum_i a_i \Big( \widehat{f}(\xi) \cdot \big( m(\xi - \lambda_i)  - m(\xi - \lambda_{i+1}) \big) \Big)^{\vee}(x) \Big| \\
    & = \left| \int_{\mathbb{R}} \left( \sum_i a_i \mathbf{1}_{(t+\lambda_i,t+\lambda_{i+1}]} \cdot \widehat{f}(\xi) \right)^{\vee}(x)  \ dm(t) \right|\\
    & \leq \| dm \|_{\text{TV}} \cdot \mathcal{V}^r( (\widehat{f}(\xi) \mathbf{1}_{(-\infty,0]}(\xi - \lambda) )^{\vee})(x),
\end{align}
from which the result follows from taking a supremum over appropriate $\{a_i\}$.
\end{proof}
We now quickly prove Lemma \ref{l:varcarZ}:
\begin{proof}[Proof of Lemma \ref{l:varcarZ}]
We cover $\mathbb{T}$ with intervals $\{ I \}$ where $|I| = 1/10$, it suffices to prove:
\begin{align}
          \sup_I \left \| \mathcal{V}^r\left( \sum_{n} f(x-n) e(\lambda n) k(n) : \lambda \in I\right) \right\|_{\ell^2(\mathbb{Z})} \lesssim_r \|f \|_{\ell^2(\mathbb{Z})}.
\end{align}
By a density argument, we may assume that $f$ is finitely supported, so that the argument inside the variation becomes continuous in $\lambda$; we may therefore restrict $\lambda$ to a countable subset of $I$, and by monotone convergence, to a finite subset $T_I \subset I$. If we let
\begin{align}
    m_\lambda(\beta) := \widehat{k}(\beta -\lambda),
\end{align}
then we may apply Magyar-Stein-Wainger transference \cite{MSW} and the above lemma. 
\end{proof}

We will use Lemma \ref{l:varcarZ} below, in our variable-coefficient variant of Bourgain's multi-frequency lemma \cite{B2}.

\subsection{Variable Coefficient Multi-Frequency Maximal Inequality}
We begin this section with an entropy estimate.

\begin{lemma}\label{l:entest}
    Suppose that $T \subset \mathbb{R}$ is finite, and let $\mathcal{H} \cong \ell^2(\Xi)$ be a finite-dimensional Hilbert space, and that 
    \[ \{ \vec{c}_t : t \in T \} \subset \mathcal{H}.\]
Suppose that $(Y,\nu)$ is a measure-space, and that $g$ is a Hilbert-space valued function,
\[ \vec{g} = (g_\xi)_{\xi \in \Xi}: Y \to \mathcal{H}\]
with
\begin{align}
    \int |\langle \vec{c},\vec{g}(y) \rangle_{\mathcal{H}}|^2 \ d\nu(y) = \int \left|\sum_{\xi} c(\xi) g_\xi(y)\right|^2 \ d\nu(y) \leq A_0^2 \| \vec{c} \, \|_{\mathcal{H}}^2
\end{align}
for all $\vec{c} \in \mathcal{H}$
and
\begin{align}
    \int \| \vec{g}(y) \|_{\mathcal{H}}^2 \ d\nu(y) = \sum_\xi \int |g_\xi(y)|^2 \ d\nu(y) \leq B_0^2.
\end{align}
Set
\[ R := 2 + \frac{B_0}{A_0};\]
for $r>s > 2$,
\begin{align}\label{e:var-entropy-conclusion}
    \Big\|
    \mathcal{V}^r\big(\langle \vec c_t,\vec g(\cdot)\rangle_{\mathcal H}:t\in T\big)
    \Big\|_{L^2(Y)}
    \lesssim
    A_0 \left( \frac{2}{s-2}+\frac{r}{r-s} \right) R^{\frac{r(s-2)}{s(r-2)}} \,
    \| \vec{c}_t \|_{\widetilde{\mathcal{V}_t^s}(\mathcal{H})}.
\end{align}
\end{lemma}
\begin{proof}
    The proof is a straightforward application of \cite[Lemma 2.6]{ZK15}, in particular we have that \[ \begin{split}
\left\| \mathcal{V}^r\left( \left \langle \vec{c}_t,\vec{g}(\cdot) \right \rangle_{\mathcal{H}}:t \in T\right)  \right\|_{L^2(Y)} &\lesssim \int_{0}^{\infty} \min \left\{ A_0 \left(\frac{\left\| \vec{c}_t \right\|_{\widetilde{\mathcal{V}_t^s}(\mathcal{H})}}{\lambda}\right)^{\frac{s}{2}},B_0 \left(\frac{\left\| \vec{c}_t \right\|_{\widetilde{\mathcal{V}_t^s}(\mathcal{H})}}{\lambda}\right)^{\frac{s}{r}}  \right\} d\lambda \\ & \lesssim \left ( \frac{2}{s-2}+\frac{r}{r-s} \right) A_0^{\frac{2(r-s)}{s(r-2)}}B_0^{\frac{r(s-2)}{s(r-2)}} \left\| \vec{c}_t \right\|_{\widetilde{\mathcal{V}_t^s}(\mathcal{H})}.\end{split}  \]
Sending $r \to \infty$ we learn that \[ \left\| \sup_{t \in T} \left| \left \langle \vec{c}_t,\vec{g}(\cdot)\right \rangle_{\mathcal{H}} \right| \right \|_{L^2(Y)} \lesssim \frac{s}{s-2} A_0^{\frac{2}{s}}B_0^{1-\frac{2}{s}} \left\| \vec{c}_t \right\|_{\widetilde{\mathcal{V}_t^s}(\mathcal{H})} \lesssim A_0\frac{s}{s-2} R^{1-2/s} \left\| \vec{c}_t \right\|_{\widetilde{\mathcal{V}_t^s}(\mathcal{H})}.  \] The latter conclusion is immediate.
\end{proof}

We now apply Lemma \ref{l:entest} in the following context.
Let $I\subset \mathbb T$ be a compact subinterval, and choose $0<c\leq |I|/100$ so small that $I+[-3c,3c]$ remains contained in a single coordinate chart of $\mathbb T$.  We choose Schwartz functions $\chi_I,\varphi_I$ such that
\begin{align}
    \widehat{\chi_I},\ \varphi_I,\ \widehat{\varphi_I}\geq 0,\qquad
    \mathbf 1_{I+[-2c,2c]}\leq \widehat{\chi_I}\leq \mathbf 1_{I+[-3c,3c]},
    \qquad \varphi_I(0)=1,
\end{align}
and $\widehat{\varphi_I}$ vanishes off $[-c,c]$.  Define the frequency multiplier
\begin{align}
    \eta_I:=\widehat{\chi_I}*\widehat{\varphi_I}.
\end{align}
Then $\eta_I\equiv1$ on $I+[-c,c]$, and its inverse Fourier transform is $\chi_I\varphi_I$.  Consequently, whenever $\widehat h$ is supported in $I$,
\begin{align}
    (\chi_I\varphi_I)*h=h;
\end{align}
we choose the dilates so that
\begin{align}
    \|\chi_I\|_{\ell^2}\lesssim |I|^{1/2}.
\end{align}

\begin{proposition}\label{p:variable}
Suppose that $2<s<r\leq \infty$, suppose that $\Xi \subset \mathbb{T}$ is finite, $\vec{g} = (g_\xi)_{\xi \in \Xi} : \mathbb{Z} \to \mathbb{C}$ and
\begin{align}
    &\sup_x \left\| \varphi_I(y) \sum_\xi g_\xi(x+y) c_\xi \right\|_{\ell^2_y} \leq A_1 \| c_\xi \|_{\ell^2(\Xi)} \qquad \text{ and} \\
    &\sup_x \left\| \left(\sum_\xi |\varphi_I(y) g_\xi(x+y)|^2\right)^{1/2} \right\|_{\ell^2_y} \leq B_1.
\end{align}
If $T \subset \mathbb{R}$ is finite, and $(K_t = m_t^{\vee})_{t \in T}$ satisfy the variational inequality
\begin{align}
    \| \widetilde{\mathcal{V}^{s}_t}(K_t*f : t \in T) \|_{\ell^2} \lesssim_s \| f \|_{\ell^2},
\end{align}
then
\[\begin{split}
    &\left\| \mathcal{V}^r \left( \sum_{\xi \in \Xi} g_\xi(x) (K_t*f_\xi)(x) : t \in T \right) \right\|_{\ell^2_x} \lesssim_{s,r}   |I|^{1/2} A_1 R^{\frac{r(s-2)}{s(r-2)}}  \left(\sum_{\xi} \|f_\xi\|_{\ell^2}^2\right)^{1/2}
\end{split}\]
whenever $\widehat{f_\xi}$ vanish off $I$ where $R \coloneqq 2+\frac{B_1}{A_1}.$ 
\end{proposition}
\begin{proof}
    By reproducing, we can express
    \begin{align}
        \sum_{\xi \in \Xi} g_\xi(x) (K_t*f_\xi)(x) = \sum_y (\chi_I \varphi_I)(y) \sum_{\xi \in \Xi} g_\xi(x) (K_t*f_\xi)(x-y),
    \end{align}
so Cauchy-Schwarz in $y$ yields the bound
\begin{align}
    &\mathcal{V}^r \left( \sum_{\xi \in \Xi} g_\xi(x) (K_t*f_\xi)(x) : t \in T \right)^2 \\
    & \lesssim \| \chi_I \|_{\ell^2}^2 \sum_y \varphi_I(y)^2 \mathcal{V}^r \left( \sum_{\xi \in \Xi} g_\xi(x) (K_t*f_\xi)(x-y) : t \in T \right)^2 \\
    & \lesssim |I| \sum_y \varphi_I(y)^2 \mathcal{V}^r \left( \sum_{\xi \in \Xi} g_\xi(x) (K_t*f_\xi)(x-y) : t \in T \right)^2.
\end{align}
Consequently, after summing in $x$ and changing variables, we obtain
\begin{align}
    &\sum_x \mathcal{V}^r \left( \sum_{\xi \in \Xi} g_\xi(x) (K_t*f_\xi)(x) : t \in T \right)^2 \\
    & \lesssim |I| \sum_x \sum_y  \mathcal{V}^r \left( \sum_{\xi \in \Xi} \left( \varphi_I(y) g_\xi(x+y) \right) \cdot (K_t*f_\xi)(x) : t \in T \right)^2 \\
    & =: |I| \sum_x \mathcal{I}(x)^2 .
\end{align}
For each fixed $x$, apply Lemma \ref{l:entest} with $\mathcal{H}=\ell^2(\Xi)$,
\begin{align}
    \vec{c}_t=(K_t*f_\xi(x))_{\xi\in\Xi},\qquad
    \vec{g}_x(y)=(\varphi_I(y)g_\xi(x+y))_{\xi\in\Xi},
\end{align}
and with 
\[ A_0=A_1, \ B_0=B_1.\]
The hypotheses of Lemma \ref{l:entest} hold by the two assumptions in the proposition, and hence
\begin{align}
    \mathcal{I}(x) \lesssim_{s,r}  A_1 R^{\frac{r(s-2)}{s(r-2)}}   \| (K_t*f_\xi(x))_{\xi\in\Xi} \|_{\widetilde{\mathcal{V}^s_t}(\ell^2(\Xi))}
\end{align}
uniformly in $x$. Taking $\ell^2_x$ norms, applying Minkowski's inequality, and using the variational hypothesis yields the bound
\[\begin{split}
  &  \left(\sum_x \mathcal{V}^r \big(\sum_{\xi \in \Xi} g_\xi(x) (K_t*f_\xi)(x) : t \in T\big)^2\right)^{1/2} 
  \\ &\lesssim_{s,r} |I|^{1/2}  A_1 R^{\frac{r(s-2)}{s(r-2)}}\left(\sum_{\xi \in \Xi} \left \| \widetilde{\mathcal{V}_t^s}(K_t*f_{\xi}(x)) \right\|_{\ell_x^2}^2  \right)^{1/2}   \lesssim_{s,r} |I|^{1/2} R^{\frac{r(s-2)}{s(r-2)}} A_1 \left( \sum_{\xi \in \Xi} \left\| f_{\xi} \right\|_{\ell^2}^2 \right)^{\frac{1}{2}}.
\end{split}\]
\end{proof}

\subsection{A Cheap Estimate}
We now quickly apply the machinery developed above to deduce a cheap $\ell^2$ bound for $\mathcal{V}^{Q,r}$, which we will upgrade by interpolation to conclude \eqref{e:red-var-l2-hyp} for all $r > 2$. In particular, we will prove below that for all $\delta>0$
\begin{align}\label{e:cheapvarbound}
    \| \mathcal{V}^{Q,r} f \|_{\ell^2} \lesssim_{r,\delta}  Q^{\delta}\|f \|_{\ell^2};
\end{align} 
this will be complemented in subsequent sections with the more complicated estimate,
\begin{align}
    \| \mathcal{C}^{Q} f \|_{\ell^2} \lesssim Q^{-\epsilon/2} \|f \|_{\ell^2}.
\end{align}

To prove \eqref{e:cheapvarbound}, we begin by abbreviating 
\[
    T_\lambda^Q f(x)
    :=
    \sum_n f(x-n)w_Q(n)K_Q(n)e(\lambda n),
\]
so that we may express
\begin{align}
    \mathcal{V}^{Q,r}f  := \mathcal{V}^r\big(T_\lambda^Q f:\lambda\in\mathbb T\big).
\end{align}

We isolate \eqref{e:cheapvarbound} in the following proposition.

\begin{proposition}\label{p:global-shell-var}
For every $r > 2$, \eqref{e:cheapvarbound} holds.
\end{proposition}

By \eqref{e:totientbound} and \eqref{e:split}, it suffices to prove
\begin{equation}
    \label{e:global-shell-var}
   \Big\|
    \mathcal{V}^r\big(T_\lambda^Q f:\lambda\in J\big)
    \Big\|_{\ell^2}   \lesssim_{r,\delta}
  Q^{\delta-1}
    \|f\|_{\ell^2}
\end{equation}
whenever $J \subset \mathbb{T}$ is an interval of size $c Q^{-2}$ with $0 <c \ll1$ chosen sufficiently small; this is since the terms involving a single $\lambda$ are controlled by the large size of the totient function, see \eqref{e:totientbound}.

So, for such a $J$, let $3J$ denote the interval concentric with $J$ but three times the length, define
\[
    \Gamma_Q+3J
    :=
    \bigcup_{a/q\in\Gamma_Q}\big(a/q+3J\big),
\]
and for each \(a/q\in\Gamma_Q\), define
\[
    \widehat{f_{a/q}}(\beta)
    := \mathbf 1_{1.1J}(\beta)
  \widehat{f}(\beta+a/q),
\]
so that \(\widehat{f_{a/q}}\) is supported in \(3J\). We now turn to the proof.


\begin{proof}[Proof of Proposition \ref{p:global-shell-var}]
Our task is to establish \eqref{e:global-shell-var}; by monotone convergence, we may restrict \(\lambda\) to a finite subset \(T \subset J\), and it suffices to replace the right-hand side
\begin{align}
    \| f \|_{\ell^2} \longrightarrow \left(\sum_{a/q \in \Gamma_Q} \| f_{a/q} \|_{\ell^2}^2\right)^{1/2}.
\end{align}
By a brief argument with the Fourier transform, whenever $\lambda \in J$ we may express
\begin{align}
    T_\lambda^Q f(x)
    &=
    \sum_{a/q\in\Gamma_Q}
    e(a x/q)
    \int_{\mathbb T}
    \widehat{F_{a/q}}(\beta)
    \widehat{K_Q}(\beta-\lambda)e(\beta x)\,d\beta, \qquad F_{a/q}:=\frac{\mu(q)}{\phi(q)}f_{a/q} \\
    & =: \sum_{a/q\in\Gamma_Q}
    e(a x/q)\,K_\lambda*F_{a/q}(x),
\end{align}
where \(K_\lambda\) is the convolution operator with Fourier multiplier
\[
    \widehat{K_\lambda h}(\beta)
    =
    \widehat{K_Q}(\beta-\lambda)\widehat h(\beta).
\]
We apply Proposition \ref{p:variable} with
\[
    \Xi=\Gamma_Q,
    \qquad
    g_{a/q}(x):=e(a x/q),
    \qquad
    I=3J;
\]
the scalar variational input follows from Lemma
\ref{l:varcarZ}, since \(\widehat{K_Q}\) has uniformly bounded variation, so it remains only to compute \(A_1\) and \(B_1\).

For any scalar sequence \(c_{a/q}\), we apply Plancherel to bound
\begin{align}
    &\sum_y\varphi_I(y)^2
    \left|
    \sum_{a/q\in\Gamma_Q}e(a(x+y)/q)c_{a/q}
    \right|^2
    \notag\\
    &\qquad=
    \sum_{a/q,a'/q'\in\Gamma_Q}
    e\big((a/q-a'/q')x\big)c_{a/q}\overline{c_{a'/q'}}
    \sum_y\varphi_I(y)^2e\big((a/q-a'/q')y\big) \\
    & \qquad \approx |I|^{-1} \sum_{a/q \in \Gamma_Q} |c_{a/q}|^2,
\end{align}
since the Fourier transform of \(\varphi_I^2\) is supported in an interval of length
\(O(|I|)\), while elements of \(\Gamma_Q\) are \(Q^{-2}\)-separated; in particular, we may bound
\[
    A_1\approx |I|^{-1/2}.
\]
Similarly,
\[
    \sum_y\varphi_I(y)^2
    \sum_{a/q\in\Gamma_Q}|e(a(x+y)/q)|^2 \lesssim Q^2 \sum_y \varphi_I(y)^2 \lesssim
    Q^2 |I|^{-1},
\]
so
\[
    B_1\lesssim Q |I|^{-1/2},
\]and thus
\[
    2+\frac{B_1}{A_1}
    \lesssim
    2+Q.
\]
Proposition \ref{p:variable}, for $s \in (2,r)$, gives
\begin{align}
    \Big\|
    \mathcal{V}^r\big(T_\lambda^Q f:\lambda\in J\big)
    \Big\|_{\ell^2}
    &\lesssim_{s,r} |I|^{1/2} A_1R^{\frac{r(s-2)}{s(r-2)}}   \big(\sum_{a/q\in\Gamma_Q}\|F_{a/q}\|_2^2\big)^{1/2},
\end{align}
which yields the result by size of the totient function and picking $s$ close enough to $2$ (depending only on $r$ and $\delta$) so that $ \frac{r(s-2)}{s(r-2)}<\frac{\delta}{2}$, see \eqref{e:totientbound}.
\end{proof}

In the next section, we will apply Proposition \ref{p:variable} to reduce our analysis to a single-scale, periodic estimate; our focus shifts to the operators
\begin{align}
    \{ \mathcal{C}^Q f : Q \}.
\end{align}

\section{Reduction to the Periodic Setting}\label{s:per}
Our next order of business will be to reduce to the periodic setting, in closer analogy to previous work on the Wiener-Wintner Theorem \cite{F+}.

We begin by partitioning the \emph{square-free} integers inside $(Q/2,Q]$ into disjoint subsets $\{ B \}$ with
\begin{align}\label{e:kappa}
    |B| \approx Q^\kappa, \; \; \; \kappa = 1/D_0 - \epsilon;
\end{align}
possibly after decreasing $\epsilon$, we will assume that
\begin{align}\label{e:epskappa} 0 < \epsilon \leq \frac{\kappa}{100}.
\end{align}

We set
\begin{align}
\mathcal{Q}_B := \text{lcm}(q \in B) \leq \exp(100 \log Q \cdot Q^\kappa),
\end{align}
noting that $\mathcal{Q}_B$ is square free, and that whenever $d|q$ for some $q \in B$
\[ (d, \frac{\mathcal{Q}_B}{d}) =1.\]

If we define
\begin{align}
    \mathcal{C}^{Q}_Bf(x):= \sup_{\lambda \in \mathbb{T}}\left| \sum_n  f(x-n) w_{Q,B}(n)K_Q(n) e(\lambda n)\right|,
\end{align}
then we may bound
\begin{align}
    \mathcal{C}^Q f(x) \leq \sum_B \mathcal{C}^{Q}_Bf(x).
\end{align}

For the remainder of the paper we will fix a single $B$, and for notational ease, set
\begin{align}
    N := \mathcal{Q}_B.
\end{align}
Regarding this parameter as fixed, we will introduce the periodic Fourier transform
\begin{align}
\mathcal{F} f(\xi) := \mathcal{F}_N f(\xi) := \mathbb{E}_{r \in [N]} f(r) e_N(-\xi r) : \mathbb{Z}/N \to \mathbb{C},
\end{align}
where we recall $e_N(t) := e(t/N)$.

The goal of this section is to reduce our study of $\mathcal{C}^Q_B$ to an analysis of the following periodic maximal functions:
\begin{align}
    M_B f(x) := \sup_{\theta \in \mathbb{Z}/N} \big| \mathbb{E}_{r \in \mathbb{Z}/N} f(x-r) w_{Q,B}(r) e_N(r \theta) \big|;
\end{align}
we define
\begin{align}
    \mathfrak m_B := \|M_B\|_{L^2(\mathbb{Z}/N)\to L^2(\mathbb{Z}/N)}
    = \sup_{\| f \|_{L^2(\mathbb{Z}/N)} = 1} \| M_B f \|_{L^2(\mathbb{Z}/N)};
\end{align}
note that by specializing the above supremum to characters, we deduce the lower bound
\begin{align}
\mathfrak m_B \geq Q^{-1}.
\end{align}

Specifically, the main result of this section is the following transference estimate.

\begin{proposition}\label{p:transferprop}
The following bound holds:
\begin{align} 
\| \mathcal{C}_B^Q f \|_{\ell^2} \lesssim_{\delta} Q^{\delta} \cdot \mathfrak m_B \cdot \|f \|_{\ell^2}.
\end{align}
\end{proposition}

To prove this proposition, we will apply the variable coefficient multi-frequency technology previously developed.

Specifically, if we let $\lambda : \mathbb{Z} \to \mathbb{T}$ be a linearizing function so that
    \begin{align}
        \mathcal{C}_B^Q f(x) \lesssim \left| \sum_n f(x-n) w_{Q,B}(n) K_Q(n) e(\lambda(x) n) \right|,
    \end{align}
    see \eqref{e:KQ},
and define 
\[ \theta : \mathbb{Z} \to \mathbb{Z}/N, \qquad \omega : \mathbb{Z} \to [-\frac{1}{2N},\frac{1}{2N}] \]
so that
\begin{align}
    \lambda(x) \equiv \frac{\theta(x)}{N} + \omega(x) \mod 1.
\end{align}
In keeping with the set up of the previous section, we will set $\Xi := \mathbb{Z}/N$, and our vector-valued function
\[ \vec{g} := (g_R)_{R \in \mathbb{Z}/N}\]
will be defined
\begin{align}
    g_R(x) := b(R- \theta(x)) e_N(Rx)
\end{align}
where our coefficients are given by
\begin{align}
b(D) := \begin{cases} \frac{\mu(q)}{\phi(q)} & \text{ if } D/N \equiv a/q, \ (a,q) = 1, q \in B \\
    0 & \text{ otherwise}. \end{cases}
\end{align}
We set 
\[ I := I_N:= \{ |\beta| \leq 1/N\},\]
and choose $\chi_I,\varphi_I$ as above, so that
\begin{align}
    &|\chi_I(y)| \lesssim_A 1/N \cdot (1 + |y|/N)^{-A},\qquad \|\chi_I\|_{\ell^2}\lesssim N^{-1/2},\\
    &0 \leq \varphi_I(y) \lesssim_A (1 + |y|/N)^{-A}, \qquad \sup_{s \in \mathbb{Z}/N} \sum_{y \equiv s \mod N} \varphi_I(y)^2 \lesssim 1.
\end{align}
Proposition \ref{p:transferprop} will follow from the below lemma, in which we compute the relevant statistics, $A_1,B_1$.
\begin{lemma}\label{l:A1B1}
    In the language of Proposition \ref{p:variable}, the following bounds hold, independent of $\theta(x)$:
\begin{align}
A_1 \lesssim  N^{1/2} \mathfrak m_B \qquad \text{ and } \qquad 
B_1 \lesssim  N^{1/2} (Q^{o(1)-1} |B|)^{1/2}.
\end{align}
In particular, for each $x \in \mathbb{Z}$, the following bounds hold:
\begin{align}
\left(\sum_y \varphi_I(y)^2 |\sum_R g_R(x+y) c_R|^2\right)^{1/2} &\lesssim N^{1/2} \mathfrak m_B \cdot \left(\sum_R |c_R|^2\right)^{1/2}  \\
\left(\sum_y \varphi_I(y)^2 \sum_R |g_R(x+y)|^2 \right)^{1/2} &\lesssim N^{1/2} (Q^{o(1)-1} |B|)^{1/2}.
\end{align}
\end{lemma}
\begin{proof}
    For the first point, let
    \[ C(s) := \sum_{R \in \mathbb{Z}/N} c_R \cdot e_N(Rs) \]
and for $\theta \in \mathbb{Z}/N$, define
\begin{align}
    L_\theta C(s) := \mathbb{E}_{r \in \mathbb{Z}/N} C(s-r) w_{Q,B}(r) e_N(\theta r)
\end{align}
so that
\begin{align}
    L_\theta C(s) = \sum_{R \in \mathbb{Z}/N} b(R-\theta) c_R e_N(Rs)
\end{align}
by orthogonality of characters. In particular,
\begin{align}
    \left| \sum_{R \in \mathbb{Z}/N} g_R(z) c_R \right| = | L_{\theta(z)} C(z)| \leq M_BC(z);
\end{align}
noting that $M_B C(z)$ is $N$-periodic, we bound
    \begin{align}
        &\sum_y \varphi_I(y)^2 \left| \sum_{R \in \mathbb{Z}/N} g_R(x+y) c_R \right|^2 \\
        & \leq \sum_{s \in \mathbb{Z}/N} \sum_{y \equiv s \mod N} \varphi_I(y)^2 M_BC(x+y)^2 \\
        & \leq \sum_{s \in \mathbb{Z}/N} M_BC(x+s \mod N)^2 \\
        & \leq \mathfrak m_B^2 \sum_{s \in \mathbb{Z}/N} |C(s)|^2 \\
        & = N \mathfrak m_B^2 \sum_R |c_R|^2. 
    \end{align}

For the second, just observe that
\begin{align}
    \sup_z \sum_R |g_R(z)|^2 &= \sum_{A \in \mathbb{Z}/N} |b(A)|^2 \\ &\lesssim Q^{o(1)-2} |\{ a/q : a \leq q, \ q \in B\}| \\
    &\lesssim Q^{o(1)-1} |B|.
\end{align}
    
\end{proof}

With this lemma in mind, we prove Proposition \ref{p:transferprop}

\begin{proof}[Proof of Proposition \ref{p:transferprop}]
Let $\rho$ be a Schwartz function satisfying
\begin{align}
    \mathbf{1}_{[-4/5,4/5]} \leq \rho \leq \mathbf{1}_{[-1,1]}.
\end{align}
Since $N=\mathcal Q_B\leq \exp(100 Q^\kappa\log Q)$ and $\kappa=1/D_0-\epsilon$, while $N(Q)=\exp(Q^{1/D_0-\epsilon/10})$, we have $N(Q)\gg_A N^A$ for all sufficiently large $Q$ (while for small $Q$ we just ensure that $\widehat{\varphi}$ is supported in a sufficiently small neighborhood of $0$); it follows that
\begin{align}
    \operatorname{supp}\widehat{K_Q}\subseteq \{\|\beta\|\leq (100N)^{-1}\}.
\end{align}
We next define
\begin{align}
    \widehat{F_R}(\beta) := \rho(N \beta) \widehat{f}(\beta + R/N)
\end{align}
    so that for each $R \in \mathbb{Z}/N$, $\widehat{F_R}$ is supported in 
    \[ I := I_N := \{ |\beta| \leq 1/N\},\]
and
\begin{align}
    \sum_R \| F_R \|_{\ell^2}^2 \lesssim \| f \|_{\ell^2}^2.
\end{align}
If we let 
\begin{align}
    m_t(\beta) := \widehat{K_Q}(\beta - t),
\end{align}
see \eqref{e:KQ},
    then
    \begin{align}
        \| \widetilde{\mathcal{V}_t^s}( m_t^{\vee}*F : t \in \mathbb{T}) \|_{\ell^2} \lesssim_s \|F \|_{\ell^2}
    \end{align}
by Lemma \ref{l:varcarZ}. If we define
\begin{align}
    g_R(x) := b(R- \theta(x)) e_N(Rx)
\end{align}
where
\begin{align}
b(D) := \begin{cases} \frac{\mu(q)}{\phi(q)} & \text{ if } D/N \equiv a/q, \ (a,q) = 1, q \in B \\
    0 & \text{ otherwise} \end{cases}
\end{align}
as above, the key claim is that
\begin{align}\label{e:keyidentity}
    \sum_n f(x-n) w_{Q,B}(n) K_Q(n) e(\lambda(x) n) = \sum_{R \in \mathbb{Z}/N} g_R(x) m_{\omega(x)}^{\vee}*F_R(x).
\end{align}
Indeed, assuming this identity, we bound
\begin{align}
    \| \mathcal{C}_B^Q f \|_{\ell^2} &\lesssim \left\| \sum_{R \in \mathbb{Z}/N} g_R(x) m_{\omega(x)}^{\vee}*F_R(x) \right\|_{\ell^2}  \lesssim_s N^{-1/2} A_1 R^{\frac{r(s-2)}{s(r-2)}}   \left(\sum_R \| F_R \|_2^2\right)^{1/2}  \\& \lesssim_s \mathfrak m_B \cdot (2 + Q)^{1-2/s} \|f \|_{\ell^2} \lesssim_{\delta} \mathfrak m_B \cdot Q^{\delta} \cdot \|f \|_{\ell^2} 
\end{align}
by Proposition \ref{p:variable}, taking into account Lemma \ref{l:A1B1}.

But, \eqref{e:keyidentity} is nothing more than algebraic manipulation: we apply Fourier inversion to the left-hand side
\begin{align}
    \eqref{e:keyidentity} &= \sum_{A \in \mathbb{Z}/N} b(A) \left( \int \widehat{f}(\beta) e(\beta x) \sum_n K_Q(n) e(-(\beta - \lambda(x) - A/N) n) \ d\beta \right) \\
    & = \sum_{A \in \mathbb{Z}/N} b(A) \left( \int \widehat{f}(\beta) e(\beta x) \widehat{K_Q}(\beta - \lambda(x) - A/N) \ d\beta \right) \\
    & = \sum_{A \in \mathbb{Z}/N} b(A) \left( \int \widehat{f}(\beta) e(\beta x) \widehat{K_Q}(\beta - \frac{\theta(x) + A}{N} - \omega(x)) \ d\beta \right) \\
    & = \int \widehat{f}(\beta) e(\beta x) \left( \sum_{A \in \mathbb{Z}/N} b(A) \widehat{K_Q}(\beta - \frac{\theta(x) + A}{N} - \omega(x)) \right) \ d\beta \\
    & = \int \widehat{f}(\beta) e(\beta x) \left( \sum_{R \in \mathbb{Z}/N} b(R - \theta(x)) \widehat{K_Q}(\beta - R/N - \omega(x)) \right) \ d\beta \\
    & = \sum_{R \in \mathbb{Z}/N} b(R-\theta(x)) e_N(R x)  \int \widehat{f}(\beta + R/N) \widehat{K_Q}(\beta - \omega(x)) e(\beta x) \ d\beta \\
    & = \sum_{R \in \mathbb{Z}/N} g_R(x) \int \widehat{f}(\beta + R/N) \widehat{K_Q}(\beta - \omega(x)) e(\beta x) \ d\beta  \\ 
    & = \sum_{R \in \mathbb{Z}/N} g_R(x) \int \widehat{F_R}(\beta) m_{\omega(x)}(\beta) e(\beta x) \ d\beta.
\end{align}
The last equality is exact: since $|\omega(x)|\leq (2N)^{-1}$ and $\operatorname{supp}\widehat{K_Q}\subseteq\{\|\beta\|\leq(100N)^{-1}\}$, the integrand vanishes unless $|\beta|\leq \frac{2}{3N}$, in which case $\rho(N\beta)=1$.

\end{proof}

For the rest of the paper, we will only be interested in estimating $L^2([N])$-normalized $f$ on $(-N,2N]$, so we can extend $f$ periodically $\mod N$, and adopt a cyclic subgroup perspective.

\subsection{Removing the Weights}
Next, we pass from our weighted periodic maximal operator, $M_B$, to simpler, unweighted ones; we appeal to convexity to do so. 

Specifically, with $d | q$ for some $q \in B$, set $D := \frac{N}{d}$, and define
\begin{align}
    M_d f(x):= \sup_{\xi \in \mathbb{Z}/D} |\mathbb{E}_{r \in [D]} f(x-rd) e_D(r \xi)|;
\end{align}
note that $M_d$ is a norm contraction on $\ell^p$ for each $1 \leq p \leq \infty$.

The operators 
\[ \{ M_d : d | q \text{ for some } q \in B\}\] control $M_B$ in the following sense: 

\begin{lemma}\label{l:MdMB}
    Suppose $g$ is $N$-periodic, and
    \begin{align}
        \| M_B g \|_{L^2( (-N,2N] )} \geq |B| Q^{-1-\epsilon} \| g \|_{L^2([N])}.
    \end{align}
Then there exists some $q \in B$, and some divisor $d|q$, so that
\begin{align}\label{e:Mdlbd}
    \| M_d g \|_{L^2( (-N,2N] )} \geq Q^{-10 \epsilon} \| g \|_{L^2([N])}.
\end{align}   
\end{lemma}
The proof is essentially an application of M\"{o}bius inversion:

Since
\begin{align}
    c_q(n) = \sum_{d|(q,n)} d \mu(q/d),
\end{align}
we can express
\begin{align}
w_B(r) = \sum_{q \in B} \frac{\mu(q)}{\phi(q)} \sum_{d |q } d \mu(q/d) \mathbf{1}_{d |r}.
\end{align}
Since each $q \in B$ is square-free, and $d | q$, $(d,q/d) =1$,
\begin{align}
    \mu(q) = \mu(d) \mu(q/d) \Rightarrow \mu(q) \mu(q/d) = \mu(d)
\end{align}
so
\begin{align}
    \frac{\mu(q)}{\phi(q)} d\mu(q/d) = \frac{d\mu(d)}{\phi(q)},
\end{align}
and thus
\begin{align}
    w_B(r) &= \sum_{q \in B} \sum_{d | q} \frac{d\mu(d)}{\phi(q)} \mathbf{1}_{d |r} \\
    & = \sum_d d \mu(d) \beta_d \mathbf{1}_{d |r},
\end{align}
where
\begin{align}\label{e:betad}
    \beta_d := \sum_{q \in B, \ d | q } \frac{1}{\phi(q)} 
\end{align}
so that
\begin{align}
    \sum_d \beta_d \leq \sum_{q \in B} \frac{1}{\phi(q)} \sum_{d|q} 1 \lesssim Q^{o(1)-1}|B|.
\end{align}

With this in mind, we present the proof.
\begin{proof}[Proof of Lemma \ref{l:MdMB}]
    We express
    \begin{align}
        &\mathbb{E}_{r \in [N]} g(x-r) w_B(r) e_N(r\theta) \\
        &= \sum_d \mu(d) \beta_d \big( \mathbb{E}_{r \in [N]} g(x-r) d \mathbf{1}_{d|r} e_N(r \theta) \big) \\
        & = \sum_d \mu(d) \beta_d \big( \mathbb{E}_{l \in [D]} g(x-ld) e_D(l \theta) \big),
    \end{align}
see \eqref{e:betad}, so
\begin{align}
    M_B g(x) &\leq \sum_d \beta_d \sup_{\theta \in \mathbb{Z}/N} \big| \mathbb{E}_{l \in [D]} g(x-ld) e_D(l \theta) \big| \\
    & \leq \sum_d \beta_d M_d g(x).
\end{align}
But now, if \eqref{e:Mdlbd} \emph{fails}, then the hypothesis gives
\begin{align}
    |B| Q^{-1-\epsilon}\|g\|_{L^2([N])} &\leq \| M_B g \|_{L^2( (-N,2N] )} \\
    &\leq \sum_d \beta_d \cdot \sup_d \|M_d g \|_{L^2( (-N,2N] )} \\
    & \leq Q^{o(1) -1} |B| Q^{-10 \epsilon} \| g \|_{L^2([N])},
\end{align}
for the desired contradiction.
\end{proof}

Accordingly, we turn our attention to $\{ M_d \}$. If we fix one such $d$, then recalling that $(d,D) = 1$, we uniquely decompose
\[ x = a + ds, \; \; \; a \in \mathbb{Z}/d, \ s \in \mathbb{Z}/D.\]
We define
\begin{align}\label{e:fF} F_a(s) := f(a + ds)
\end{align}
and compute:
\begin{align}
    \widehat{F_a}(\xi) := \mathbb{E}_{s\in[D]} F_a(s) e_D(-\xi s), \; \; \; \xi \in \mathbb{Z}/D;
\end{align}
with this notation in mind, we establish the following convenient representation of $M_d$ in the below lemma.

\begin{lemma}
    For every $a \in \mathbb{Z}/d$ and every $s \in \mathbb{Z}/D$, and every $\xi \in \mathbb{Z}/D$,
    \begin{align}
        \mathbb{E}_{r \in [D]} f(a + d(s-r)) e_D(r \xi) = e_D(\xi s) \widehat{F_a}(\xi);
    \end{align}
consequently,
\begin{align}
    M_df(a+ds) = \sup_{\xi \in \mathbb{Z}/D} |\widehat{F}_a(\xi)|.
\end{align}
    \end{lemma}
\begin{proof}
This follows since we have extended $f$ to be periodic $\mod N$, and thus $F_a(s)$ is periodic $\mod D$:
\begin{align}
    \mathbb{E}_{r \in [D]} f(a + d(s-r)) e_D(r \xi) &= \mathbb{E}_{r \in [D]} F_a(s-r) e_D(r\xi) \\
    & = e_D(s \xi) \mathbb{E}_{t \in s- [D]} F_a(t) e_D(-t \xi) \\
    & = e_D(s \xi) \mathbb{E}_{t \in [D]} F_a(t) e_D(-t \xi)
\end{align}
since 
\[ t \mapsto F_a(t) e_D(-\xi t) \]
is $D$-periodic.
\end{proof}

With this representation in mind, we next prove an ``inverse theorem" for $M_d$, which we will extend to $M_B$ by suitably applying Lemma \ref{l:MdMB}.

To do so, we consider families of ``structured atoms," 
\begin{align}
    \mathcal{A}_d := \{  g(a) e_D(\vartheta(a) s), \; \; \; a \in \mathbb{Z}/d, \ \vartheta(a), s \in \mathbb{Z}/D, \ \mathbb{E}_{a \in [d]} |g(a)|^2 = 1 \},
\end{align}
and collect
\begin{align}
\mathcal{A} := \bigcup_{q \in B} \bigcup_{d | q} \mathcal{A}_d;
\end{align}
we have the following.
\begin{lemma}[Inverse Theorem for $M_d$]\label{l:invMd}
    Suppose that $f : \mathbb{Z}/N \to \mathbb{C}$ is $N$-periodic. Then one can choose $\Phi \in \mathcal{A}_d$ so that
    \begin{align}
        \| M_d f \|_{L^2((-N,2N])} = |\langle \Phi, f \rangle_{(-N,2N]}|,
    \end{align}
see \eqref{e:Innerprod}. In particular, if
    \begin{align}
        \| M_d f \|_{L^2((-N,2N])} \geq \lambda \| f \|_{L^2([N])},
    \end{align}
    there exists $\Phi \in \mathcal{A}_d$ so that
    \begin{align}
        |\langle \Phi, f \rangle_{(-N,2N]}| \geq \lambda \|f \|_{L^2([N])}.
    \end{align}
\end{lemma}
\begin{proof}
    Define
    \begin{align}
        M(a) := \sup_{\xi \in \mathbb{Z}/D} |\widehat{F_a}(\xi)|,
    \end{align}
see \eqref{e:fF}; note that the only way that $M(a) \equiv 0$ is if $f \equiv 0$.
    
Then, since $M(a+ds) = M(a)$
\begin{align}
    \mathbb{E}_{a \in (-N,2N]} |M(a)|^2 = \mathbb{E}_{a \in [d]} \mathbb{E}_{s \in (-D,2D]} |M(a)|^2 = \mathbb{E}_{a \in [d]} |M(a)|^2.
\end{align}
Next, for each $a \in [d]$, choose $\vartheta(a) \in \mathbb{Z}/D$ so that
\begin{align}
    |\widehat{F_a}(\vartheta(a))| = |M(a)|.
\end{align}
If we set
\begin{align}
    g(a) := \frac{\widehat{F_a}(\vartheta(a))}{\| M_d f \|_{L^2((-N,2N])} }
\end{align}
and
\begin{align}
    \Phi(a+ds) := g(a) e_D(\vartheta(a) s)
\end{align}
    then $\Phi \in \mathcal{A}_d$ as
    \begin{align}
        \mathbb{E}_{a \in [d]} |g(a)|^2 = \frac{\mathbb{E}_{a \in [d]} |M(a)|^2 }{ \| M_d f \|_{L^2( (-N,2N] )}^2} = 1.
    \end{align}
We compute
\begin{align}
    &\langle f(a+ds), \Phi(a+ds) \rangle_{(-N,2N] }\\
    & = \mathbb{E}_{a \in [d]} \overline{g(a)} \mathbb{E}_{s \in (-D,2D]} F_a(s) e_D(-\vartheta(a)s) \\
    & = \mathbb{E}_{a \in [d]} \overline{g(a)} \widehat{F_a}(\vartheta(a)) \\
    & = \frac{1}{\| M_d f \|_{L^2( (-N,2N] )}} \cdot \mathbb{E}_{a \in [d]} |M(a)|^2 \\
    &= \|M_d f \|_{L^2( (-N,2N])}.
\end{align}
    
\end{proof}

Combining Lemmas \ref{l:MdMB} and \ref{l:invMd}, we arrive at the following corollary.
\begin{cor}\label{c:invMB}
    Set $X := Q^{-1-\epsilon} |B|$ and $\eta := Q^{-10\epsilon}$. If $g$ is $N$-periodic and satisfies
    \begin{align}
        \| M_B g \|_{L^2((-N,2N])} \geq X \|g \|_{L^2([N])},
    \end{align}
then there exists $\Phi \in \mathcal{A}$ so that
\begin{align}
    |\langle g, \Phi \rangle_{ (-N,2N] }| \geq \eta \| g \|_{L^2([N])}.
\end{align}
    \end{cor}

We now convert Corollary \ref{c:invMB} to a structure theorem via an energy decrement argument; we maintain the notation
\[ X := Q^{-1-\epsilon} |B|, \; \; \; \eta := Q^{-10 \epsilon}.\]

\begin{proposition}[Greedy Decomposition]\label{p:greedy}
Suppose $\| f \|_{L^2([N])} = 1$ is $N$-periodic. Then there exists a $J \leq Q^{25 \epsilon}$ so that
\[ f = \sum_{j \leq J} c_j \Phi_j + g,\]
where
\begin{align}
    \| M_B g \|_{L^2( (-N,2N] )} \leq X, \; \; \; \;\;\;\sum_{j \leq J} |c_j|^2 \leq 1
\end{align}
and $\{ \Phi_j \} \subset \mathcal{A}$ are structured atoms.
\end{proposition}
\begin{proof}
    Define $f_0 := f$, and abbreviate
    \[ \langle h_1,h_2 \rangle_{(-N,2N]} := \langle h_1,h_2 \rangle;\]
    we will recursively construct a sequence of functions,
    \[ \| f_{k+1} \|_{L^2([N])} \leq \| f_k \|_{L^2([N])} \leq \dots \leq \|f\|_{L^2([N])} \leq 1\]   
that will terminate at time $J \leq Q^{25 \epsilon}$, and a sequence of coefficients $\sum_{j \leq J} |c_j|^2 \leq 1$, so that for
\[ g := f_{J}, \qquad J \leq Q^{25 \epsilon}\]
we may estimate
    \begin{align}
        \| M_B g \|_{L^2( (-N,2N] )} := \| M_B \big( f - \sum_{j \leq J} c_j \Phi_j \big) \|_{L^2( (-N,2N] )} \leq X.
    \end{align}
To do so, we use the absolute stopping rule
    \begin{align}\label{e:query}
        \| M_B f_k \|_{L^2( (-N,2N] )} \leq X.
    \end{align}
    If this holds, set $g := f_k$; 
    if \eqref{e:query} fails, then, since \(\|f_k\|_{L^2([N])}\le1\),
we have
\[
\|M_Bf_k\|_{L^2((-N,2N])}>X\ge X\|f_k\|_{L^2([N])},
\]
so the hypothesis of Corollary \ref{c:invMB} holds. Hence there exists a structured atom $\Phi_{k+1} \in \mathcal{A}$ so that
    \begin{align}
        |\langle f_k, \Phi_{k+1} \rangle| \geq \eta \|f_k \|_{L^2([N])}.
    \end{align}
Choose the phase of $\Phi_{k+1}$ so that
\begin{align}
    c_{k+1} := \langle f_k,\Phi_{k+1} \rangle \geq \eta \| f_k \|_{L^2([N])},
\end{align}
and define
\begin{align}
    f_{k+1} := f_k - c_{k+1} \Phi_{k+1}.
\end{align}
Then, because $\Phi_{k+1}$ is normalized and $c_{k+1}=\langle f_k,\Phi_{k+1}\rangle$,
\begin{align}
    \| f_{k+1}\|_{L^2([N])}^2 = \|f_k\|_{L^2([N])}^2 - |c_{k+1}|^2,
\end{align}
so
\begin{align}
    \| f_{k+1} \|_{L^2([N])}^2 \leq (1-\eta^2) \|f_k \|_{L^2([N])}^2,
\end{align}
and more generally
\begin{align}\label{e:geomdecay}
    \| f_k \|_{L^2([N])}^2 \leq (1-\eta^2)^k.
\end{align}
We now \textbf{STOP} this recursion the first time that \eqref{e:query} holds; note that
\begin{align}
    \sum_{j=1}^J |c_j|^2 = \|f \|_{L^2([N])}^2 - \|g \|_{L^2([N])}^2 \leq 1. 
\end{align}
If we let $A := |B| Q^{o(1)-1}$ denote the trivial operator norm bound for $M_B$, then whenever $\| g \|_{L^2([N])} \leq Q^{-10}$, 
\begin{align}
    \| M_B g \|_{L^2((-N,2N])} \leq A \| g \|_{L^2([N])} \leq X;
\end{align}
given the geometric decay \eqref{e:geomdecay}, this constraint is certainly satisfied for some $J \leq Q^{25\epsilon}$.
\end{proof}

So, matters reduce to analyzing structured atoms, in particular to producing significant gain on 
\[ M_B \Phi \]
when $\Phi \in \mathcal{A}$; we will essentially produce a square-root savings in $|B| \approx Q^\kappa$ which will close the argument.

\section{Analysis of Structured Atoms}\label{s:number}
Taking into account Proposition \ref{p:red-var}, Lemma \ref{l:red-var-l2}, Proposition \ref{p:transferprop}, and mixed-norm interpolation, see \cite[\S 7]{Kold}, the following estimate will close our argument.
\begin{proposition}\label{p:goal0}
    The following bound holds, uniformly in $B$:
    \begin{align}
        \mathfrak m_B \lesssim Q^{\kappa - 1 -\epsilon}.
    \end{align}
\end{proposition}

We prove Proposition \ref{p:goal0} by analyzing the behavior of $M_B$ when applied to elements of $\mathcal{A}$.  In fact, the corrected structured-atom estimate gives the full bound $Q^{-1+o(1)}$ for each atom.  We then take a trivial $\ell^1$ sum over the coefficients in our atomic representation and conclude, since $\epsilon = \epsilon(p)>0$ is chosen sufficiently small relative to $\kappa$.

For notational ease, set
\begin{align}
    w_B := w_{Q,B} := \sum_{q \in B} \frac{\mu(q)}{\phi(q)} c_q.
\end{align}
For each $d | N, \ d \leq Q$ and $\theta \in [N]$, define
\begin{align}
T_\theta f(x) := \mathbb{E}_{u \in [d]} f(x-u) K_\theta(u)   
\end{align}
where
\begin{align}
K_\theta(u) := \mathbb E_{l\in[D]} w_B(ld+u) e_N((dl+u)\theta);
\end{align}
our argument will be concluded with a suitably delicate analysis of 
\[ \{ K_\theta : \theta \in [N]\}.\]

Below, we regard $d$ as fixed; note that since $N$ is square-free, $(d,D) =1$.

We collect
\begin{align}
    B_t :=\{ q \in B : \frac{q}{(q,d)} = t \}, \; \; \;  t_\theta := \frac{D}{(\theta,D)}, \; \; \; \mathcal{T} := \{ t | D : B_t \neq \emptyset \};
\end{align} 
our first lemma is elementary:
\begin{lemma}\label{l:elemestat}
The following estimates hold:
\begin{align}
\sup_t |B_t| \lesssim Q^{o(1)}, \; \; \; |\mathcal{T}| \leq |B|.   
\end{align}
\end{lemma}
\begin{proof}
For the first point, express $r = (q,d)$ so $q = rt$ where $(r,t) = 1$. For fixed $t$, once $r$ is chosen, there's only one possible choice of $q$, and there are at most $\tau(d) \leq Q^{o(1)}$ many choices of $r$.

For the second, just observe that each $q \in B$ contributes exactly one value 
\[ \frac{q}{(q,d)} = t.\]
\end{proof}

We have the following characterization.

\begin{lemma}\label{l:periodicfiber}
   $K_\theta \equiv 0$ unless $t_\theta \in \mathcal{T}$. And, in the other case,
   \begin{align}
   K_\theta(u) = e_d( (\overline{D})_d \theta u) L_{t_\theta}(u)
\end{align}
where $D (\overline{D})_d \equiv 1 \mod d$,
and
\begin{align}
    L_t(u) := \sum_{q \in B_t} \frac{\mu(q)}{\phi(q)} c_{(q,d)}(u).
\end{align}
\end{lemma}
The significance of this lemma is that the kernels $\{ L_t \}$ have very small operator norms, in that, by Lemma \ref{l:periodicfiber} and Lemma \ref{l:elemestat}, 
\begin{align}\label{e:supLest}
    \mathbb{E}_{u \in [d]} \sup_t  |L_t(u)| &\lesssim Q^{o(1)-1 } \mathbb{E}_{u \in [d]} \max_{r | d} |c_r(u)| \\
    & \lesssim Q^{o(1)-1} \mathbb{E}_{u \in [d]} \phi\big((d,u)\big) \\
    &\lesssim Q^{o(1) -1},
\end{align}
where we used the square-free nature of $d$ to simplify
\[ |c_r(u)|=\phi((r,u)) \qquad r|d,\]
and therefore
\begin{align}
    \max_{r|d}|c_r(u)|=\phi\big((d,u) \big).
\end{align}

\begin{proof}[Proof of Lemma \ref{l:periodicfiber}]
Express
    \begin{align}
    K_\theta(u) = \mathbb{E}_{l \in [D]} w_B(ld+u) e_N((dl + u)\theta)  = \sum_{q \in B} \frac{\mu(q)}{\phi(q)} A_q(\theta,u),
\end{align}
where
\begin{align}
    A_q(\theta,u) &:= \mathbb{E}_{l \in [D]} c_q(ld + u) e_N( (ld + u ) \theta) \\
    & = \sum_{(a,q) = 1} \mathbb{E}_{l \in [D]} e_N( (Na/q + \theta)(ld+u) )\\
    &= \sum_{(a,q) =1} e_N((Na/q + \theta)u) \cdot \mathbb{E}_{l \in [D]} e_N( (Na/q + \theta) dl) \\
    &= \sum_{(a,q) =1} e_N((Na/q + \theta)u) \cdot \mathbb{E}_{l \in [D]} e_D( (Na/q + \theta) l) \\
    & = \sum_{\substack{(a,q) =1 \\ N\frac{a}{q} + \theta \equiv 0 \mod D}} e_N( (Na/q + \theta) u ).
\end{align}

Since $(d,D) = 1$, every divisor $q | N$ factors uniquely as
\[ q = q_1q_2, \; \; \; q_1 = (q,d), \; \; \; q_2 = (q,D),\]
where $(q_1,q_2) = 1$, since $q$ is square-free. So,
\begin{align}
    \frac{N}{q} = \frac{d}{q_1} \cdot \frac{D}{q_2},
\end{align}
and thus the congruence relation becomes
\begin{align}
    \theta + \frac{d}{q_1} \frac{D}{q_2} a \equiv 0 \mod D,
\end{align}
which implies, after reducing $\mod \frac{D}{q_2}$, that
\begin{align}
    \frac{D}{q_2} | \theta \Rightarrow \theta = \frac{D}{q_2} t
\end{align}
for some $t \in \mathbb{Z}$.

Dividing by $\frac{D}{q_2}$ and re-arranging, we see that
\begin{align}
    t + \frac{d}{q_1} a \equiv 0 \mod q_2.
\end{align}
The claim now is that $\frac{d}{q_1}$ is invertible $\mod q_2$: this is since $q_2|D$, and $(d,D) = 1$; we conclude that there is a unique residue $\alpha \in [q_2]$ so that
\begin{align}
    a \equiv \alpha \equiv - \overline{ (\frac{d}{q_1}) }_{q_2} \cdot t \mod q_2,
\end{align}
where
\begin{align}
    \frac{d}{q_1} \cdot \overline{ (\frac{d}{q_1}) }_{q_2} \equiv 1 \mod q_2.
\end{align}

Finally, since $(a,q) = 1 \Rightarrow (a,q_2) = 1$, we must have 
\begin{align}
    (\alpha,q_2) = ( \overline{ (\frac{d}{q_1}) }_{q_2} \cdot  t , q_2) = 1 \Rightarrow (t,q_2) = 1.
\end{align}

To summarize, if $K_\theta \neq 0$, we must have
\begin{align}
\frac{D}{q_2} | \theta =: \frac{D}{q_2} t \qquad \text{and} \qquad ( \frac{\theta}{D/q_2},q_2) = (t,q_2) = 1; \end{align}
this means
\[ (\theta,D) = \frac{D}{q_2}(t,q_2) = \frac{D}{q_2} \Rightarrow q_2 = (q,D) = \frac{q}{(q,d)} = \frac{D}{(\theta,D)}.\]

In particular, we have shown that there must exist an $a, \ (a,q) =1$ with
\[ D | (Na/q + \theta) \iff (q,D) = \frac{D}{(\theta,D)}\]
and thus $A_q(\theta,u)$ vanishes unless
\[ q_2 = (q,D) = \frac{D}{(\theta,D)}.\]

With this in mind, below we reduce our sum over $a, \ (a,q) =1$ to only those so that
\[ a \equiv \alpha \mod q_2,\]
where $\alpha$ is the unique residue class determined above, and in particular satisfies $(\alpha,q_2) =1$.

So, assume that
\[ q_2 = (q,D) = \frac{q}{(q,d)} = \frac{D}{(\theta,D)},\]
and compute $A_q(\theta,u)$.

\begin{align}
    A_q(\theta,u) = \mathbf{1}_{q_2 = \frac{D}{(\theta,D)}} \cdot \sum_{\substack{(a,q) =1 \\ N\frac{a}{q} + \theta \equiv 0 \mod D}} e_N( (Na/q + \theta) u );  
\end{align}
we have seen that the set of $a$ so that $(a,q) =1$ so that the sum does not vanish must satisfy $a \equiv \alpha \mod q_2$ for some unique $\alpha$ with
\begin{align}
    \frac{d}{q_1} \alpha + \frac{\theta}{D/q_2} \equiv 0 \mod q_2
\end{align}
and
\begin{align}
a \equiv x \mod q_1, \; \; \; (x,q_1) = 1    
\end{align}
for some unique $x$.

If we express
\begin{align}
    \frac{N}{q} a + \theta = D m_a,
\end{align}
then
\begin{align}
e_N((\frac{N}{q} a + \theta)u) = e_d(m_au),
\end{align}
and our job is to compute $m_a \mod d$. Since $D$ is coprime to $d$, $(D,d) =1$, its multiplicative inverse $\mod d$, $(\overline{D})_d$, is well-defined. Then, working $\mod d$:
\begin{align}
m_a \equiv (\overline{D})_d \cdot (\frac{N}{q} a + \theta) \equiv (\overline{D})_d \cdot ( \frac{d}{q_1} \frac{D}{q_2} a + \theta)  \equiv \frac{d}{q_1} (\overline{q_2})_d \cdot a + (\overline{D})_d \cdot \theta,
\end{align}
where $q_2 (\overline{q_2})_d \equiv 1 \mod d$, and $(\overline{q_2})_d$ is well defined, since $(q_2,d) = 1$.

Since
\begin{align}
     \frac{d}{q_1} (\overline{q_2})_d \cdot a \mod d
\end{align}
depends only on $a \mod q_1$, i.e.\ on $x$, working $\mod d$, we may express
\begin{align}
    m_a \equiv \frac{d}{q_1} (\overline{q_2})_d \cdot a + (\overline{D})_d \cdot \theta \equiv \frac{d}{q_1} (\overline{q_2})_d \cdot x + (\overline{D})_d  \cdot \theta \mod d,
\end{align}
so
\begin{align}
    e_d(m_a u) &= e_d((\overline{D})_d  \cdot \theta u)\cdot  e_d(\frac{d}{q_1} (\overline{q_2})_d  \cdot  x u) \\
    & = e_d((\overline{D})_d  \cdot\theta u) \cdot e_{q_1}((\overline{q_2})_d  \cdot x u);
\end{align}
since multiplication by $(\overline{q_2})_d$ permutes $(x,q_1) = 1$, we have
\begin{align}
    \sum_{(x,q_1) =1} e_{q_1}((\overline{q_2})_d  \cdot x u) = c_{q_1}(u) = c_{(q,d)}(u).
\end{align}
And, since we have derived a unique constraint on $a \mod q_2$, this shows that \begin{align}
    A_q(\theta,u) &= e_d((\overline{D})_d  \cdot \theta u) \cdot c_{(q,d)}(u)\cdot  \mathbf{1}_{(q,D)= D/(\theta,D)} \\
    &= e_d((\overline{D})_d  \cdot\theta u) \cdot  c_{(q,d)}(u) \cdot \mathbf{1}_{q/(q,d) = D/(\theta,D)};
\end{align}
the upshot is that we can express
    \begin{align}
        K_\theta(u) = e_d((\overline{D})_d  \cdot \theta u) \sum_{\substack{q \in B \\ q/(q,d) = D/(\theta,D)}} \frac{\mu(q)}{\phi(q)} c_{(q,d)}(u),
    \end{align}
which completes the proof.
\end{proof}

With this representation in mind, we quickly conclude the argument in the following subsection.
\subsection{The Maximal Function on Structured Atoms}
\begin{proposition}\label{p:structuredatom}
Suppose that
\[ f(a+ds) = g(a) e_D(\vartheta(a) s) \in \mathcal{A}, \]
namely $\mathbb{E}_{a \in [d]} |g(a)|^2 = 1$. Then
\begin{align}
    \| M_B f \|_{L^2( (-N,2N] )} \lesssim Q^{o(1)-1}.
\end{align}
\end{proposition}
\begin{proof}
For $\rho\in\mathbb Z/D$ define
\begin{align}
    \widetilde K_\rho(u):=\mathbb E_{\ell\in[D]}w_B(u+d\ell)e_D(\ell\rho).
\end{align}
If $\rho\in\mathbb Z/D$ and $\Theta\in\mathbb Z/N$ is any lift of $\rho$, then
\begin{align}
    K_\Theta(u)=e_N(u\Theta)\widetilde K_\rho(u).
\end{align}
Lemma \ref{l:periodicfiber} therefore implies the pointwise envelope
\begin{align}\label{e:Henv}
    |\widetilde K_\rho(u)|\leq H_d(u):=\sup_{t\in\mathcal T}|L_t(u)|
\end{align}
for every $\rho\in\mathbb Z/D$.

Fix $\Theta\in\mathbb Z/N$ and write
\begin{align}
A_\Theta f(x):=\mathbb E_{r\in\mathbb Z/N}f(x-r)w_B(r)e_N(r\Theta).
\end{align}
If we let $x=a+ds$, write $r=u+d\ell$, and let $a_u=[a-u]_d$ be the residue class of $a-u$ modulo $d$, we may choose $\epsilon(a,u)\in\mathbb Z$ so that
\begin{align}
    a-u=a_u-d\epsilon(a,u),
\end{align}
and thus
\begin{align}
    a+ds-u-d\ell=a_u+d(s-\epsilon(a,u)-\ell),
\end{align}
so
\begin{align}
&A_\Theta f(a+ds) \\
&=\mathbb E_{u\in[d]}g(a_u)e_D(\vartheta(a_u)(s-\epsilon(a,u)))e_N(u\Theta)
    \mathbb E_{\ell\in[D]}w_B(u+d\ell)e_D(\ell(\Theta-\vartheta(a_u))) .
\end{align}
Using \eqref{e:Henv}, we obtain the pointwise bound
\begin{align}
    |A_\Theta f(a+ds)|\leq \mathbb E_{u\in[d]}|g(a-u)|H_d(u),
\end{align}
where $g$ is extended periodically modulo $d$.  Since the right hand side is independent of $s$, Young's inequality on $\mathbb Z/d$ gives
\begin{align}
\|M_Bf\|_{L^2(\mathbb Z/N)}
&\leq \big\|\mathbb E_{u\in[d]}|g(\cdot-u)|H_d(u)\big\|_{L^2(\mathbb Z/d)}\\
&\leq \|g\|_{L^2(\mathbb Z/d)}\mathbb E_{u\in[d]}H_d(u) \\
& =\mathbb E_{u\in[d]}H_d(u) \\
& \leq Q^{o(1)-1}
\end{align}
by \eqref{e:supLest}.
\end{proof}

In the language of Proposition \ref{p:greedy}, we conclude as follows, taking into account \eqref{e:epskappa} above. 
\begin{cor}
    Suppose that $\| f \|_{L^2([N])} = 1$, and that
        \begin{align}
        f = \sum_{j \leq J} c_j \Phi_j + g
    \end{align}
    is the decomposition of Proposition \ref{p:greedy}.
    Then
    \begin{align}
      \| M_B f \|_{L^2((-N,2N])} &\leq   \| M_B ( \sum_{j \leq J} c_j \Phi_j ) \|_{L^2((-N,2N])} + \| M_B g \|_{L^2((-N,2N])} \\
      &\lesssim Q^{-1+15\epsilon} + Q^{\kappa - 1- \epsilon} \\
      & \lesssim Q^{\kappa - 1 -\epsilon}.
    \end{align}
\end{cor}
This establishes Proposition \ref{p:goal0}, and with it, Theorems \ref{t:main} and \ref{t:carmax}.

\appendix

\section{Complements}\label{a:comp}

\subsection{On Admissibility}
In \cite{FKT1,FKT2}, a category of \emph{admissible} weights were introduced, for which weighted double recurrence and return times theorems were established:

A weight, $w \geq 0$ is \emph{admissible} if it satisfies the following properties:

    \begin{itemize}
        \item \textbf{Upper normalization:} There exists an absolute constant $0 < C < \infty$ so that for any $0 \leq C_0 < \infty$, 
    \begin{align}
        \limsup_{N \to \infty} \, \frac{1}{N} \sum_{n \in (C_0N,C_0N +N]} w(n) \leq C;
    \end{align}
        \item \textbf{Heath-Brown-type approximation}: There exists an absolute constant $C < \infty$, a function 
        \[ S \coloneqq S_w \colon \mathbb{Q} \to \mathbb{C} \text{ so that } |S(a/q)| \leq C \cdot q^{o(1)-1},\]
        and a constant $\nu > 4$, so that for some $s \geq 3$ and all sufficiently large $N$, there exists a truncation parameter $M = M(N) \leq N^{o(1)}$ so that
        \begin{align}
            \left\| w(n) - \sum_{q\leq M}\sum_{\substack{1\leq a\leq q\\ (a,q)=1}} S(a/q) e^{2 \pi i n a/q} \right\|_{U^s([N])} \leq C \cdot \log^{-\nu/2^s} N.
        \end{align}
    \end{itemize}

Below we explain why this condition is \emph{insufficient} to establish a weighted Carleson-type theorem:

\begin{lemma}
For $0 < \alpha < 1$, and an integer $q_0 \geq 3$, let
\begin{align}
    v(n) := v_{\alpha,q_0}(n) := 1 + \alpha \sin(2 \pi n/q_0).
\end{align}
Then $v$ is admissible, but
\begin{align}
    \mathcal{C}_v f(x) := \sup_\lambda \big|\sum_n f(x-n) v(|n|) \frac{e(\lambda n)}{n} \big|
\end{align}
is unbounded on $\ell^p$ for each $1 \leq p \leq \infty$.
\end{lemma}
\begin{proof}
Observe that, since $0 < \alpha < 1$, $v \geq 0$; by periodicity its averages on every interval whose length tends to $\infty$ are uniformly bounded, and we can express
\begin{align}
    v(n) := 1 + \frac{\alpha}{2i} e(n/q_0) - \frac{\alpha}{2i} e(n(q_0-1)/q_0)
\end{align}
so the Heath-Brown approximation is satisfied \emph{exactly}. To prove that $\mathcal{C}_v$ is unbounded, we apply dimensional analysis, testing it on
\[ f(x) := f_L(x) := \mathbf{1}_{[-L,L]};\]
note that whenever $|x| \leq L/4$, for each $|m| \leq 3L/4$, both $x\pm m \in [-L,L]$.

With this in mind, choosing $\lambda = \frac{1}{q_0}$, we compute
\begin{align}
    \mathcal{C}_v f(x) &\geq \left| \sum_{1 \leq |n| \leq 3L/4} f_L(x-n) v(|n|) \frac{e(n/q_0)}{n}\right| - O(1) \\
    &= \left| \sum_{1 \leq m \leq 3L/4} v(m) \frac{e(m/q_0) - e(-m/q_0)}{m}\right| - O(1) \\
    & = \left| 2i \sum_{ 1 \leq m \leq 3L/4} \frac{\sin(2\pi m/q_0)}{m} + 2i \alpha \sum_{1 \leq m \leq 3L/4} \frac{\sin^2(2\pi m/q_0)}{m} \right| - O(1) \\
    & = 2 \alpha \left| \sum_{1 \leq m \leq 3L/4} \frac{\sin^2(2\pi m/q_0)}{m} \right| - O_{q_0}(1),
\end{align}
where we applied summation by parts to address the first sum. Since $q_0 \geq 3$, the frequency $\frac{2}{q_0} \notin \mathbb{Z}$, so we express
\begin{align}
    \sum_{1 \leq m \leq 3L/4} \frac{\sin^2(2\pi m/q_0)}{m} &= \frac{1}{2} \sum_{1 \leq m \leq 3L/4} \frac{1}{m} - \frac{1}{2} \sum_{1 \leq m \leq 3L/4} \frac{\cos(4 \pi m/q_0)}{m} \\
    & =\frac{1}{2} \log L + O(1)
\end{align}
by summation by parts.

In particular
\begin{align}
    \mathcal{C}_v f(x) \gtrsim \alpha \log L \cdot \mathbf{1}_{|x| \leq L/4}
\end{align}
which yields the result.
\end{proof}

More generally, suppose a real periodic model has Fourier coefficients with
\[ S(-a/q) = \overline{S(a/q)};\]
its sine component at $a/q$ is non-zero exactly when $S(a/q) \notin \mathbb{R}$, in which case adapting the preceding proof produces a logarithmic divergence. In particular, a natural symmetry hypothesis is
\[ S(-a/q) = S(a/q) \in \mathbb{R};\]
note that this condition is satisfied in the case of the primes, in which case
\[ S(a/q) = \frac{\mu(q)}{\varphi(q)}.\]

\subsection{Extension to Piatetski-Shapiro Times}
Our methods are sufficiently flexible to extend to sparse subsets of the primes, namely to Piatetski-Shapiro primes, i.e.\ those primes which are also elements of the following set:
\[\mathbb{N}_c := \{ k = \lfloor n^c \rfloor \text{ for some } n \in \mathbb{N} \}.\]
In particular, if we let
\begin{align}
    w_c(n) := w(n) \cdot c |n|^{1-1/c} \cdot \mathbf{1}_{\mathbb{N}_c}(|n|)
\end{align}
denote the natural density on the Piatetski-Shapiro primes, then we may replace
\[ w\longrightarrow w_c, \; \; \; 1 \leq c < 7/6\]
in the statement of Theorem \ref{t:main}. For notational ease, we sketch the argument for the maximal formulation below:

Let 
\[ W_c(n) := w(n) \cdot ( 1 - c |n|^{1-1/c} \mathbf{1}_{\mathbb{N}_c}(|n|));\]
by \cite{FKT1} we know that there exists an $\epsilon > 0$ so that
\begin{align}
\| \sup_\lambda | \mathbb{E}_{n \in [N]} f(x-n) W_c(n) e(\lambda n) | \|_{\ell^2} \lesssim N^{-\epsilon} \|  f \|_{\ell^2}.    
\end{align}

\begin{lemma}\label{l:Pscar}
 The following operator is $\ell^p$ bounded for each $1 < p <\infty$:
 \begin{align}
    \sup_{\lambda \in \mathbb{T}} \left|\sum_{n \neq 0} f(x-n) w(n) c |n|^{1-1/c} \mathbf{1}_{\mathbb{N}_c}(|n|) \cdot \frac{e(\lambda n)}{n} \right| 
 \end{align}
\end{lemma}
\begin{proof}
By the boundedness of the Carleson operator along the primes, it suffices to bound
     \begin{align}
    \sup_{\lambda \in \mathbb{T}} \left|\sum_{n \neq 0} f(x-n) W_c(n) \frac{e(\lambda n)}{n}\right |. 
 \end{align}
    
We bound
    \begin{align}
\left\| \sup_{\lambda \in \mathbb{T}} \left|\sum_n  f(x-n) W_c(n) \psi_k(n) e(\lambda n)\right|\right \|_{\ell^2} \lesssim \|W_c\|_{U^3([-2^k,2^k])} \|f \|_{\ell^2} \lesssim 2^{-\epsilon k } \|f\|_{\ell^2};
    \end{align}
since the single scale estimate is bounded on $\ell^p$ by convexity, we interpolate to conclude. 
\end{proof}
To adapt this argument to the variational setting, it suffices to note that 
\begin{align}
    \| W_c \|_{L^2([-2^k,2^k])} \lesssim 2^{k/2(1-1/c)}
\end{align}
and then to argue by interpolation as in \S \ref{s:red}. 

\section{The Discrete Carleson Operator}\label{a:disc}
We here provide a brief proof of the following discrete analogue.
\begin{proposition}[Discrete Carleson's Theorem]\label{p:disc}
For each $r > 2$, and $r' < p < \infty$, the following bound holds:
    \begin{align}
     \left   \| \mathcal{V}^r\left( \sum_{n \neq 0 } f(x-n) \frac{e(\lambda n)}{n} : \lambda \in \mathbb{T} \right) \right\|_{\ell^p} \lesssim_{r,p}  \| f\|_{\ell^p}
    \end{align}
\end{proposition}
\begin{proof}
 Let $\{ \psi_j \}$ be as above; which by arguing as in \eqref{e:smoothing} we may assume that  $\supp(\widehat{\psi_k})\subseteq[-1/4,1/4]$ and by a limiting argument  there is no loss of generality in reducing to
    \begin{align}
    \left    \| \mathcal{V}^r \left( \sum_n  f(x-n) \sum_{1 \leq k \leq J} \psi_k(n) e(\lambda n) : \lambda \in \mathbb{T} \right) \right\|_{\ell^p} \lesssim_{r,p}  \| f\|_{\ell^p},
    \end{align}
provided our estimates are independent of $J$. 

Now, for an appropriately normalized Schwartz function $\eta$  with $\supp(\eta) \subseteq \left[-\frac{3}{4},\frac{3}{4}\right]$ define
\begin{align}
    \widehat{f_l}(\beta) := (1 - |\beta|)_+ \cdot \eta(2^{10} \beta -l) \cdot \widehat{f}(\beta),
\end{align}
so that
\begin{align}
    \sum_{|l| \leq 2^{100}} f_l \equiv f
\end{align}
as sequence-space functions, by an argument with the F\'{e}jer kernel. And,
\begin{align}
    \sup_l \| f_l \|_{\ell^p} \lesssim \| f \|_{\ell^p}
\end{align}
by convexity. It therefore suffices to establish Proposition \ref{p:disc} for each $f_l$ individually; since the problem is modulation invariant, we may reduce to the $l=0$ case. By \eqref{e:split} and the triangle inequality, we may assume that $\frac{r}{100} \leq \lambda < \frac{r+1}{100}$ for some $0 \leq r < 100$.

If we pass to Fourier space, and apply Poisson summation, we can express
\begin{align}
    &\sum_{n}\sum_{1 \leq k \leq J} f_0(x-n) \psi_k(n) e(\lambda n)  \\
    & = \int \sum_v M( v + (\beta - \lambda)) (1 - |\beta|)_+ \eta(2^{10} \beta) \widehat{f}(\beta) e(\beta x) \ d\beta,
\end{align}
where
\begin{align}
    M(t) := \sum_{1 \leq k \leq J} \widehat{\psi_k}(t).
\end{align}
And, since we have restricted $\beta,\lambda$, we see that the sum over $v$ collapses to a single term, with $v \in \{0, \pm 1\}$ according to the particular value of $r$. So, our task becomes estimating
\begin{align}
   \mathcal{V}^r\left( \int M_{v}(\beta, \lambda) \widehat{f}(\beta) e(\beta x) \ d\beta :  \frac{r}{100} \leq \lambda < \frac{r+1}{100} \right)
\end{align}
where
\begin{align}
   M_{v}(\beta,\lambda) := M( v + \beta - \lambda ) (1 - |\beta|)_+ \eta(2^{10} \beta), \; \; \; v \in \{0, \pm 1\}.
\end{align}
Since $f$ is finitely supported, by continuity we may restrict the above supremum in $\lambda$ to a countable set, so by monotone convergence, it suffices to prove
\begin{align}
 \left  \|   \mathcal{V}^r\left( \int M_{v}(\beta, \lambda) \widehat{f}(\beta) e(\beta x) \ d\beta : \lambda \in T \right) \right\|_{\ell^p} \lesssim_{r,p}  \| f\|_{\ell^p},
\end{align}
uniformly over finite subsets $T \subset [\frac{r}{100},\frac{r+1}{100})$. But this just follows from Magyar-Stein-Wainger transference \cite{MSW}.
\end{proof}

We remark that the same argument applies whenever the Hilbert kernel, $\frac{1}{t}$, is replaced by any Calder\'{o}n-Zygmund kernel whose Fourier transform has bounded variation, and a periodic formulation exists as well.

\begin{cor}\label{c:disc}
Let $Q \geq 1$ be an integer, and let $K$ be a Calder\'{o}n-Zygmund kernel with $\| \widehat{K} \|_{\text{BV}} = 1$.
Then for each $r > 2, \ r' < p < \infty$, the following bound holds:
    \begin{align}
     \left   \| \mathcal{V}^r \left( \sum_{n \neq 0 } f(x-Q n) K(n) e(\lambda n) : \lambda \in \mathbb{T} \right)\right \|_{\ell^p} \lesssim_{r,p}  \| f\|_{\ell^p},
    \end{align}
    with the implicit constant independent of $Q$.
\end{cor}

\end{document}